\documentclass[11pt]{article}

\usepackage[a4paper,margin=1in]{geometry}
\usepackage[numbers,sort&compress]{natbib}
\usepackage{microtype}
\usepackage{amsmath,amssymb,amsfonts,amsthm,mathrsfs,mathtools}
\usepackage{algorithm}
\usepackage{algpseudocode}
\usepackage{booktabs}
\usepackage{multirow}
\usepackage{array}
\usepackage{longtable}
\usepackage{graphicx}
\usepackage{subcaption}
\usepackage{xcolor}
\usepackage{enumitem}
\usepackage{flafter}
\usepackage[section]{placeins}
\usepackage{accents}
\usepackage[colorlinks=true,linkcolor=blue,citecolor=blue,urlcolor=blue,hypertexnames=false]{hyperref}

\graphicspath{{figures/}}
\allowdisplaybreaks[4]
\numberwithin{equation}{section}
\newtheorem{theorem}{Theorem}[section]
\newtheorem{assumption}[theorem]{Assumption}

\newtheorem{lemma}[theorem]{Lemma}

\theoremstyle{definition}
\newtheorem{definition}[theorem]{Definition}
\newtheorem{example}[theorem]{Example}
\theoremstyle{remark}
\newtheorem{remark}[theorem]{Remark}

\newcommand{\ubar}[1]{\underaccent{\bar}{#1}}
\renewcommand{\leq}{\leqslant}
\renewcommand{\geq}{\geqslant}
\def\epsilon{\varepsilon}
\def\phi{\varphi}

\title{Adaptive Bregman Proximal Stochastic Gradient with a Stabilized Barzilai--Borwein Step Size}
\author{
Chenhan Jin$^{1}$ \quad Shengze Xu$^{1}$ \quad Binghui Xie$^{1}$ \quad Kaiwen Zhou$^{1}$\\
Fan JIA$^{2}$ \quad James Cheng$^{1}$ \quad
Tieyong Zeng$^{3,4}$\thanks{Corresponding author: \texttt{tieyongzeng@bnbu.edu.cn}}\\[0.5em]
$^{1}$The Chinese University of Hong Kong \quad
$^{2}$University of Utah\\
$^{3}$Beijing Normal-Hong Kong Baptist University \quad
$^{4}$Guangzhou Nanfang College
}
\date{}

\begin{document}

\maketitle

\begin{abstract}
Bregman proximal stochastic gradient (BPSG) methods bring variance-reduced composite optimization to objectives whose geometry is poorly captured by Euclidean smoothness. Their performance, however, remains sensitive to the step size: raw stochastic curvature estimates can fluctuate sharply, whereas line searches add repeated proximal evaluations. We introduce Ada-BPSG, a line-search-free BPSG method that couples the SAGA gradient table with a stabilized Barzilai--Borwein (BB) candidate. A mediant aggregates incremental secant information so that nearly singular local ratios receive little weight, and an explicit safeguard translates the resulting curvature estimate into the bounded step-size sequence required for convergence. This design yields a direct analytical chain from relative smoothness and component-wise variance control to convergence in finite-dimensional normed spaces. We prove an $O(n/K)$ ergodic rate for convex objectives, a restarted linear rate under relative quadratic growth, and an $O(1/K)$ bound for a Bregman proximal residual in the nonconvex setting. On logistic regression and sparse nonnegative matrix factorization, Ada-BPSG combines low objective values with substantially less sensitivity to the initial step size than standard variance-reduced baselines, while avoiding line search.
\end{abstract}

\section{Introduction}
We study finite-sum composite optimization on a finite-dimensional normed vector space $\mathbf{E}$:
\begin{equation}\label{Problem}
    \min_{x\in\mathbf{E}}\ F(x):=f(x)+h(x),
    \qquad
    f(x):=\frac{1}{n}\sum_{i=1}^{n}f_i(x),
\end{equation}
where $[n]:=\{1,\ldots,n\}$ indexes the component losses $f_i:\mathbf{E}\to\mathbb{R}$, and $h:\mathbf{E}\to(-\infty,+\infty]$ is a proper convex regularizer whose proximal subproblem is tractable. This model includes regularized empirical risk minimization, constrained learning, and structured estimation. When $h=0$ and $\mathbf{E}=\mathbb{R}^d$, stochastic gradient descent (SGD) samples $i_k\in[n]$ and updates $x_{k+1}=x_k-\eta_k\nabla f_{i_k}(x_k)$ with step size $\eta_k>0$.

For a nonsmooth $h$, the Euclidean proximal stochastic-gradient update replaces the gradient step by
\begin{equation} \label{SGD_update}
    x_{k+1}
    =\arg\min_x
    \left\{
        \langle\widetilde{\nabla}_k,x\rangle
        +\frac{1}{2\eta_k}\|x-x_k\|_2^2+h(x)
    \right\},
\end{equation}
where $\widetilde{\nabla}_k$ is an unbiased estimator of the full gradient $\nabla f(x_k)$ conditional on the iteration history. With a constant step size, the persistent variance of a plain sampled gradient can prevent the iterates from settling near a solution. Two complementary responses are therefore central to stochastic optimization: variance-reduced (VR) estimators, which make the gradient noise vanish, and adaptive step sizes, which respond to the local scale of the objective. The first family includes SAG, SAGA, SVRG, and SARAH \citep{DBLP:conf/nips/RouxSB12,DBLP:conf/nips/DefazioBL14,DBLP:conf/nips/Johnson013,DBLP:conf/icml/NguyenLST17}; the second includes the Barzilai--Borwein (BB) rule, Polyak step sizes, AdaGrad, and Adam \citep{barzilai1988two,polyak1964some,DBLP:journals/jmlr/DuchiHS11,DBLP:journals/corr/KingmaB14}. Their combination has produced adaptive VR methods such as SVRG-BB, SAG-BB, AdaSVRG, AI-SARAH, ADASPIDER, AdProxGD, and 2AdaSPIDER \citep{DBLP:conf/nips/TanMDQ16,DBLP:journals/corr/abs-2102-09645,DBLP:journals/corr/abs-2102-09700,kavis2022adaptive,malitsky2024adaptive,huang2025faster}.

The Euclidean model behind \eqref{SGD_update} can nevertheless be a poor match for modern objectives. A globally Lipschitz Euclidean gradient is unavailable in applications such as neural-network training, quantum tomography, and matrix or tensor factorization \citep{hasannasab2020parseval,kuang2019fusion,comon2008symmetric}. Bregman geometry replaces the squared Euclidean distance by a divergence generated by a problem-adapted kernel $\psi$:
\begin{equation} \label{BPSG_update}
    x_{k+1}
    =\arg\min_x
    \left\{
        \langle\widetilde{\nabla}_k,x\rangle
        +\frac{1}{\eta_k}D_\psi(x,x_k)+h(x)
    \right\}.
\end{equation}
Section \ref{prelimi} formally defines the Bregman divergence $D_\psi$. The kernel encodes the geometry of the domain, while relative smoothness replaces a global Euclidean Lipschitz-gradient condition \citep{doi:10.1137/16M1099546,bauschke2017descent,bolte2018first}.

This geometric extension sharpens the step-size problem. Bregman divergences are generally asymmetric, so Euclidean polarization identities no longer drive the descent proof, and stochastic secant ratios may vary markedly across sampled components. Diminishing schedules trade this instability for slow progress, while line searches add repeated objective or proximal evaluations. Our goal is to retain stochastic curvature adaptation without either cost.

\subsection{Related work and positioning}
Bregman projection, forward--backward, mirror-prox, accelerated Bregman proximal-gradient, and relative-smooth first-order methods establish the deterministic geometric foundation of this work \citep{bauschke1997legendre,butnariu2000totally,van2017forward,bui2021bregman,doi:10.1137/16M1099546,hanzely2021accelerated,bauschke2017descent,bolte2018first}. In Euclidean geometry, optimized proximal-gradient schemes have also been applied to structured inverse problems, including $\ell_1$- and total-variation-regularized non-blind image deblurring \citep{wang2025improved}. Their stochastic counterparts incorporate sampled gradients, primal--dual splitting, or variance reduction \citep{silvetifalls2021stochastic,nguyen2023stochastic,wang2023bregman,wang2024bregman}. In particular, BPSG-SAGA and BPSG-SARAH extend Bregman proximal optimization to variance-reduced estimators, and extrapolated BPSG uses line search to control the resulting updates \citep{wang2023bregman,wang2024bregman}.

Adaptive stochastic methods approach the same stability question through the step size. AdaGrad and Adam accumulate gradient statistics \citep{DBLP:journals/jmlr/DuchiHS11,DBLP:journals/corr/KingmaB14,DBLP:conf/icml/WardWB19}; universal and relative-smoothness schemes estimate local regularity in mirror-descent, primal--dual, or extragradient frameworks \citep{levy2018online,kavis2019unixgrad,bach2019universal,cohen2021relative,antonakopoulos2019adaptive,antonakopoulos2021adaptive,antonakopoulos2022undergrad}; and BB methods extract scalar curvature from secant pairs \citep{barzilai1988two}. Stabilized BB and BB-VR variants improve this estimate in stochastic Euclidean problems \citep{DBLP:conf/nips/TanMDQ16,burdakov2019stabilized,ma2018stochastic,yang2023improved,zhou2024adabb}. Related powered and double-adaptive methods use BB or SPIDER-style estimators in Euclidean finite-sum and minimax settings \citep{yang2023improved,huang2021efficient,huang2025faster}.

Ada-BPSG connects these lines through one mechanism: the SAGA table supplies incremental secant pairs, a stabilized BB rule converts them into a curvature-sensitive candidate, and clipping converts that candidate into a provably stable Bregman proximal step. Table \ref{tab:positioning} places it among the closest Bregman proximal methods. This leads to the central question of the paper:

\begin{table}[t]
\centering
\small
\setlength{\tabcolsep}{5pt}
\caption{Positioning of Ada-BPSG among Bregman proximal (variance-reduced) methods. All operate on a distance-generating kernel $\psi$. ``LS-free'' marks methods needing no line search; ``Needs $\bar L,M$'' lists the smoothness constants required to set the step or safeguard. Rates are ergodic orders in the iteration count $K$ at fixed batch size; ``lin.\ (RQG)'' is the restarted linear rate under relative quadratic growth. A dash marks a regime the cited analysis does not cover: the stochastic Bregman baselines \citep{wang2023bregman,wang2024bregman} are established only in the nonconvex setting, so the convex ergodic and restarted-linear guarantees are, to our knowledge, new to Ada-BPSG. Crucially, our analysis is carried out in a general finite-dimensional normed space---using only the dual norm, with no Euclidean polarization or norm decomposition---so the matching $O(n/K)$ order is obtained under strictly weaker geometric assumptions rather than by improving the rate. Over BPSG-SAGA, Ada-BPSG adds only two scalar accumulations per refreshed component.}
\label{tab:positioning}
\begin{tabular}{lcccccc}
\toprule
Method & Oracle & Step-size rule & LS-free & Needs $\bar L,M$ & Convex & Nonconvex\\
\midrule
BPG \citep{bolte2018first} & full & fixed $1/\bar L$ & Yes & $\bar L$ & $O(1/K)$ & $O(1/K)$\\
BPSG-SAGA \citep{wang2023bregman,wang2024bregman} & SAGA & fixed & Yes & $\bar L,M$ & -- & $O(1/K)$\\
BPSGE \citep{wang2024bregman} & SAGA/SARAH & line search & No & $\bar L,M$ & -- & $O(1/K)$\\
\textbf{Ada-BPSG (ours)} & SAGA & safeguarded BB & Yes & $\bar L,M$ & $O(n/K)$; lin.\ (RQG) & $O(1/K)$\\
\bottomrule
\end{tabular}
\end{table}

\emph{Can a line-search-free BPSG method adapt to stochastic curvature while preserving convergence in non-Euclidean composite optimization?}

\subsection{Contributions}
We answer this question with Ada-BPSG. Its contributions are threefold.
\begin{itemize}
    \item \textbf{A stable adaptive step from SAGA information.}\par We aggregate recent component-wise secant pairs by a mediant, equivalently a denominator-weighted mean of local BB ratios. The weighting suppresses ratios associated with nearly singular curvature estimates. Clipping and a monotone update then turn this raw candidate into a bounded, line-search-free step-size sequence. For linear-prediction losses, the same quantities can be maintained with an $O(n)$ scalar table instead of the general $O(nd)$ SAGA table.

    \item \textbf{A convergence analysis aligned with the algorithm.}\par Relative smoothness controls Bregman descent, local component smoothness controls the SAGA table variance, and the safeguard closes the resulting Lyapunov recursion. The entire argument is carried out in a general finite-dimensional normed space, relying only on the dual norm, dual-norm Young inequalities, and the three-point Bregman inequality: it never invokes the Euclidean polarization identity or the norm decomposability that Euclidean adaptive variance-reduced analyses rest on. This chain yields an $O(n/K)$ ergodic rate for convex objectives, a restarted linear rate under relative quadratic growth, and an $O(1/K)$ bound on the expected squared Bregman proximal residual for nonconvex objectives. Matching these Euclidean-SAGA orders under the weaker non-Hilbert geometry---rather than improving them---is the technical content of the analysis. The convex ergodic and restarted-linear guarantees moreover have no counterpart in the SAGA/SARAH-based stochastic Bregman methods we build on, which are analyzed only in the nonconvex regime \citep{wang2023bregman,wang2024bregman}.

    \item \textbf{Experiments that test the design mechanism.}\par A curvature diagnostic shows why arithmetic averages of stochastic BB ratios can spike and how the mediant stabilizes the candidate. Across four logistic-regression datasets, Ada-BPSG reaches low objective gaps in fewer effective passes and remains substantially less sensitive to initialization than standard VR baselines. On a simplex-constrained Poisson inverse problem under the entropy kernel---an instance meeting every hypothesis of the analysis, in which variance reduction is genuinely needed---it exhibits the predicted convex rate in genuinely non-Euclidean geometry and, by adapting past the conservative worst-case safeguard, improves on every fixed-step baseline held to its a-priori guaranteed step by more than two orders of magnitude---a gain that carries over unchanged to a region-pooled Poisson unmixing of the real Samson hyperspectral scene; on sparse nonnegative matrix factorization it attains the lowest objective curves among the compared Bregman methods while avoiding the line search used by the accelerated baseline.
\end{itemize}

\section{Algorithm}
\label{s:algorithm}
Ada-BPSG adds one curvature-adaptive scalar to the BPSG-SAGA update. The construction proceeds from the classical BB secant ratio, identifies why an arithmetic average of stochastic ratios remains unstable, and replaces that average by a safeguarded mediant computed from information already stored by SAGA.

\subsection{Barzilai--Borwein step size}
\label{s:algorithm-deterministic}
The BB method is a scalar quasi-Newton scheme: it approximates the local Hessian by $\eta_k^{-1}I$ and chooses $\eta_k$ to fit a secant equation. In the Euclidean setting, define the iterate and gradient differences
$s_k:=x_k-x_{k-1}$ and $y_k:=\nabla f(x_k)-\nabla f(x_{k-1})$. Minimizing either $\|\eta_k^{-1}s_k-y_k\|_2$ or $\|s_k-\eta_k y_k\|_2$ gives the two classical rules \citep{barzilai1988two}
\begin{equation} \label{BBmethods}
    \eta_k^{\mathrm{BB1}}=\frac{s_k^Ts_k}{s_k^Ty_k}
    \qquad\text{and}\qquad
    \eta_k^{\mathrm{BB2}}=\frac{s_k^Ty_k}{y_k^Ty_k}.
\end{equation}
The first and second expressions correspond to the two residual problems, respectively. Both estimate inverse curvature along the latest displacement; our construction follows the BB2 orientation because SAGA directly supplies component-gradient differences.

\subsection{Why stochastic BB ratios need stabilization}
A BB ratio uses only one displacement and is therefore sensitive to sampled curvature. In the BB2 rule, a small $y_k^Ty_k$ can produce an excessively large step, especially near a solution or along a low-curvature component. Denominator regularization and explicit bounds alleviate this behavior \citep{fletcher2005barzilai,burdakov2019stabilized,ma2018stochastic,zhou2024adabb}, but an arithmetic mean over recent ratios can still be dominated by a single nearly singular denominator.

Figure \ref{fig: local information} isolates this failure mode on the \emph{ijcnn1} dataset \citep{chang2011libsvm}. For the regularized logistic component
\[
    f_i(x)=\log(1+\exp(-b_i a_i^Tx))+\frac{\lambda}{2}\|x\|_2^2,
\]
$a_i\in\mathbb{R}^d$ is a feature vector and $b_i\in\{-1,+1\}$ is its label. The Hessian at epoch $k$ is
\[
    H_k=\frac1n\sum_{i=1}^n p_i(x_k)a_ia_i^T+\lambda I,
    \qquad
    p_i(x):=\frac{\exp(-b_i a_i^Tx)}{[1+\exp(-b_i a_i^Tx)]^2},
\]
so $p_i(x)$ records the sample-dependent curvature weight. Even as these weights settle, the vanilla and arithmetic-mean stabilized BB rules in Figure \ref{fig: local information-2} spike to $10^3$--$10^4$. In contrast, the mediant candidate stays within the smallest and smoothest range among the plotted rules. The diagnostic motivates a division of labor that runs through the paper: the mediant stabilizes the empirical curvature estimate, and the safeguard enforces the bounds used by the convergence analysis.

\subsection{Stabilized Barzilai--Borwein candidate}
\label{s:algorithm-saga-bb}

\begin{figure}[t!]
    \centering
    \begin{subfigure}[b]{0.46\textwidth}
        \centering
        \includegraphics[width=\linewidth]{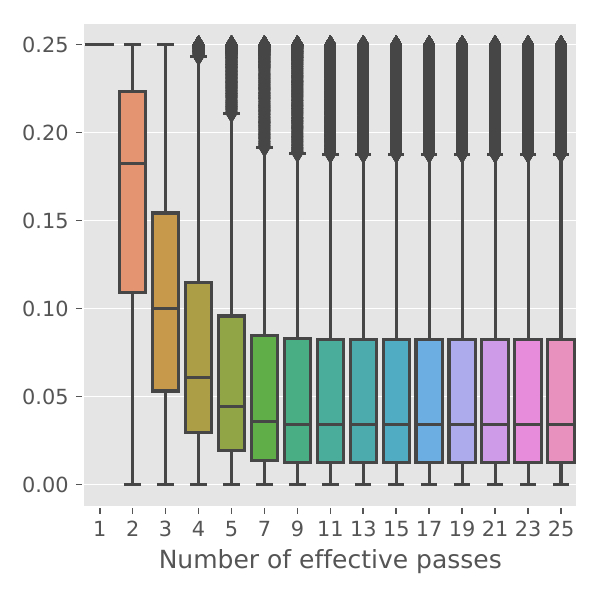}
        \caption{}\label{fig: local information-1}
    \end{subfigure}
    \hfill
    \begin{subfigure}[b]{0.46\textwidth}
        \centering
        \includegraphics[width=\linewidth]{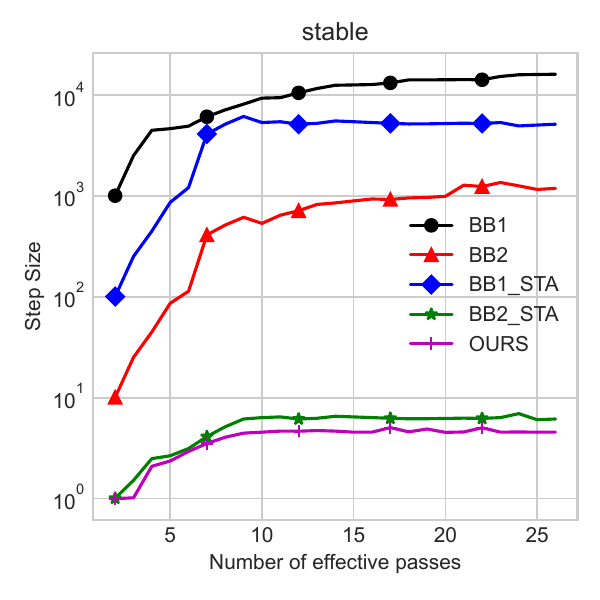}
        \caption{}\label{fig: local information-2}
    \end{subfigure}
    \caption{Stochastic curvature and BB step sizes on \emph{ijcnn1}, averaged over four runs. Left: distribution of the component curvature weights $p_i(x_k)$. Right: step-size trajectories over 30 effective passes; the mediant-based rule suppresses the spikes produced by vanilla and arithmetic-mean BB variants.}
    \label{fig: local information}
\end{figure}
For positive denominators, the mediant of ratios satisfies
\[
    \frac{\sum_i a_i}{\sum_i b_i}
    =
    \sum_i \frac{b_i}{\sum_j b_j}\frac{a_i}{b_i}.
\]
Hence the mediant is a denominator-weighted arithmetic mean: a nearly singular denominator contributes little weight instead of dominating the aggregate.

To apply this idea to SAGA, let $\phi_i^t$ denote the most recent point stored for component $i$ at iteration $t$, and let $\mathcal W_k$ collect the component updates since the previous step-size refresh. Each record $(i,t)\in\mathcal W_k$ supplies the secant pair
\[
    s_{i,t}:=\phi_i^{t+1}-\phi_i^t,
    \qquad
    y_{i,t}:=\nabla f_i(\phi_i^{t+1})-\nabla f_i(\phi_i^t).
\]
Algorithm \ref{alg:stc_schm} maintains the two window sums
\[
    \xi_k:=\sum_{(i,t)\in\mathcal W_k}|\langle y_{i,t},s_{i,t}\rangle|,
    \qquad
    \delta_k:=\sum_{(i,t)\in\mathcal W_k}\|y_{i,t}\|_*^2,
\]
where $\langle\cdot,\cdot\rangle$ is the duality pairing and $\|\cdot\|_*$ is the dual norm. The raw adaptive candidate is
\begin{equation} \label{stableBB}
    \widehat{\eta}_{k}^{\text{BB}}
    =
    \begin{cases}
    \dfrac{1}{\alpha}\dfrac{\xi_k}{\delta_k+\theta\xi_k}, & \xi_k>0,\\[4pt]
    \eta_{k-1}, & \xi_k=0.
    \end{cases}
\end{equation}
The absolute value makes the numerator nonnegative when a nonconvex component produces a negative secant pairing. The buffer $\theta>0$ prevents denominator degeneracy and bounds the raw candidate by $1/(\alpha\theta)$; $\alpha>0$ controls its scale. In Euclidean space, the construction reduces to an aggregated BB2-type ratio.

The safeguard maps this curvature estimate to
\begin{equation} \label{clippedBB}
    q_k=\operatorname{clip}_{[\eta_{\min},\eta_{\max}]}
    (\widehat{\eta}_{k}^{\text{BB}})
    :=
    \min\{\eta_{\max},\max\{\eta_{\min},\widehat{\eta}_{k}^{\text{BB}}\}\},
\end{equation}
where $0<\eta_{\min}\leq\eta_{\max}$. Algorithm \ref{alg:stc_schm} then uses the monotone update
\[
    \eta_k=\max\{\eta_{k-1},q_k\},
\]
which preserves the interval and supplies the monotonicity needed by the convex Bregman telescope. Thus the raw ratio governs adaptation within the interval, while the safeguard supplies the stability invariant used in Section \ref{s:theory}.

\subsection{Main method}
Algorithm \ref{alg:stc_schm} integrates the candidate into BPSG-SAGA. A batch $B_k$ refreshes the corresponding table entries, the standard SAGA correction forms the unbiased estimator $\widetilde{\nabla}_k$, and the same table differences update $\xi_k$ and $\delta_k$. These accumulators are reset every $m$ iterations after refreshing the step size, so adaptation requires only scalar reductions beyond the gradient-table operations.

The parameters have separate, explicit roles: $m$ is the refresh frequency, $\theta$ regularizes the denominator, $\alpha$ rescales the raw ratio, and $[\eta_{\min},\eta_{\max}]$ defines the stability range. The algorithm adapts within this range instead of following a prescribed diminishing schedule. When $\mathbf{E}=\mathbb{R}^d$ and $\psi(x)=\frac12\|x\|_2^2$, the Bregman update becomes the Euclidean proximal update \eqref{SGD_update}; we call this specialization SAGA-BB in the experiments.

\begin{algorithm}[t]
\caption{Ada-BPSG}
\label{alg:stc_schm}
\small
\begin{algorithmic}[1]

\State \textbf{Input:} max iterations $K$, batch-size $b\in\{1,\ldots,n\}$, initial point $x_0$, kernel $\psi$
\Statex \hspace{\algorithmicindent} initial step-size $\eta_0\in[\eta_{\min},\eta_{\max}]$, update frequency $m\geq1$
\Statex \hspace{\algorithmicindent} BB scaling $\alpha>0$, BB buffer $\theta>0$, safeguard bounds $0<\eta_{\min}\leq\eta_{\max}$
\State \textbf{Initialize:} $\phi_i^{0}=x_0$, gradient table $\{\nabla f_i(\phi_i^0)\}_{i=1}^n$
\Statex \hspace{\algorithmicindent} $\mu_0=\frac{1}{n}\sum_{i=1}^n\nabla f_i(\phi_i^{0})$, $\xi_0=0$, and $\delta_0=0$

\For{$k = 0,1,\dots,K-1$}

    \If{$k>0$}
        \State $\eta_k=\eta_{k-1}$
    \EndIf

    \If{$k>0$ and $k \bmod m = 0$}
        \State $\widehat{\eta}_k =
        \begin{cases}
        \frac{1}{\alpha}\frac{\xi_k}{\delta_k + \theta\xi_k}, & \xi_k>0,\\
        \eta_{k-1}, & \xi_k=0,
        \end{cases}$
        \State $\eta_k=\max\{\eta_{k-1},
        \operatorname{clip}_{[\eta_{\min},\eta_{\max}]}(\widehat{\eta}_k)\}$
        \State $\xi_k = 0$, $\delta_k = 0$
    \EndIf

    \State Randomly pick a batch $B_k$ uniformly without replacement
    \State For $i \in B_k$, set $\phi_i^{k+1}=x_k$ and update $\nabla f_i(\phi_i^{k+1})$
    \State For $i\notin B_k$, set $\phi_i^{k+1}=\phi_i^k$
    
    \State $\widetilde{\nabla}_{k} =
        \frac{1}{b}\sum_{i \in B_k}
        \left(\nabla f_i(\phi_i^{k+1})-\nabla f_i(\phi_i^{k})\right)
        + \mu_k$
    
    \State $x_{k+1}
        = \arg\min_{x}\left\{
            h(x) + \langle \widetilde{\nabla}_k, x\rangle
            + \frac{1}{\eta_k} D_{\psi}(x,x_k)
        \right\}$

    \State Compute $y_i = \nabla f_i(\phi_i^{k+1}) - \nabla f_i(\phi_i^{k})$ for $i \in B_k$
    \State Compute $s_i = \phi_i^{k+1} - \phi_i^{k}$ for $i \in B_k$

    \State $\xi_{k+1} = \xi_k + \sum_{i\in B_k} |\langle y_i, s_i\rangle|$
    \State $\delta_{k+1} = \delta_k + \sum_{i\in B_k} \|y_i\|_*^2$

    \State $\mu_{k+1} = \mu_k + \frac{1}{n}
        \sum_{i\in B_k}
        \left(\nabla f_i(\phi_i^{k+1})-\nabla f_i(\phi_i^{k})\right)$

\EndFor

\State \textbf{return} $x_K$

\end{algorithmic}
\end{algorithm}

\subsection{Efficient implementation}
The general SAGA table stores one $d$-dimensional gradient per component and therefore costs $O(nd)$ memory. Linear-prediction losses have the form $f_i(x)=\Psi_i(\langle a_i,x\rangle)$, where $a_i\in\mathbb{R}^d$ is a data vector and $\Psi_i:\mathbb{R}\to\mathbb{R}$ is a scalar loss \citep{zhou2019direct}. In this case, storing only $\langle a_i,\phi_i^k\rangle$ recovers both the gradient-table correction and the secant sums in \eqref{stableBB}. The implementation in Appendix A.2 consequently reduces the table memory to $O(n)$ without changing the Ada-BPSG update.

This reduction follows from the factorization
\[
    \nabla f_i(\phi_i^k)
    =\Psi_i'(\langle a_i,\phi_i^k\rangle)a_i,
\]
where $\Psi_i'$ denotes the derivative of the scalar loss. When component $i$ is refreshed, its old and new scalar predictions determine both the SAGA gradient correction and its contributions to $\xi_k$ and $\delta_k$; the feature norm $\|a_i\|_*$ can be precomputed. Ada-BPSG therefore needs neither a separate curvature table nor a history of past iterates.

The same reuse preserves the computational profile of the underlying BPSG-SAGA iteration. For a mini-batch of size $b$, step-size adaptation adds two scalar accumulations per refreshed component and a constant-cost safeguard whenever the step size is updated. It introduces no extra component-gradient evaluations or proximal subproblems. This is the implementation counterpart of the method's central design: the stochastic correction and the curvature estimate use the same local information, while the safeguard converts that information into the stable step-size sequence required by the analysis.

\section{Convergence Analysis}
\label{s:theory}
The analysis follows the same chain as the algorithm. Relative smoothness converts the Bregman proximal update into descent, component smoothness controls the SAGA estimation error, and the safeguard keeps the adaptive step within the range that balances these two terms. We first define the geometry and the step-size invariant, then establish variance bounds for the convex and nonconvex regimes, and finally derive the convergence rates.

\subsection{Preliminaries} \label{prelimi}
\paragraph{Normed-space notation.}
Let $\mathbf{E}$ be a finite-dimensional real normed vector space with norm $\|\cdot\|$, and let $\mathbf{E}^*$ be its dual space, the space of continuous linear functionals on $\mathbf{E}$. For $y\in\mathbf{E}^*$ and $x\in\mathbf{E}$, $\langle y,x\rangle$ denotes their duality pairing, and
\[
    \|y\|_*:=\max\{\langle y,x\rangle:\ \|x\|\leq 1\}.
\]
This expression defines the dual norm $\|\cdot\|_*$. The filtration $\mathcal{F}_k:=\sigma(B_0,\ldots,B_{k-1})$ contains the mini-batches sampled before iteration $k$, and $\mathbb{E}_k[\cdot]:=\mathbb{E}[\cdot\mid\mathcal{F}_k]$ denotes conditional expectation given this history.

Since $\mathbf{E}$ is finite-dimensional, it is automatically complete, hence Banach; closed bounded subsets are compact, all norms are equivalent, and $\mathbf{E}^*$ has the same dimension. Every existence and boundedness fact used below---solvability of the proximal subproblem and existence of $x^*\in D$---follows from $\dim\mathbf{E}<\infty$ together with the strong convexity of $\psi$, so no completeness or sequential-compactness assumption beyond finite-dimensionality is invoked. In fact the convergence results are non-asymptotic, finite-horizon bounds (Theorems \ref{theo1}--\ref{theo3}) and use no limit-point or subsequence argument at all. The generality relative to the Euclidean setting is therefore geometric---the norm need not derive from an inner product, so no polarization identity or norm decomposition is available---rather than topological.

\paragraph{Bregman geometry.}
A distance-generating function (DGF), or Bregman kernel, is a continuously differentiable function $\psi:\mathbf{E}\to\mathbb{R}$ that is 1-strongly convex with respect to $\|\cdot\|$:
\[
    \psi(x)-\psi(y)-\langle \nabla\psi(y),x-y\rangle\geq \frac{1}{2}\|x-y\|^2.
\]
Its Bregman divergence is the first-order linearization error
\[
    D_{\psi}(x,y):=\psi(x)-\psi(y)-\langle \nabla\psi(y),x-y\rangle.
\]
Strong convexity gives $D_\psi(x,y)\geq\frac12\|x-y\|^2$, which connects Bregman descent to the norm differences used in the SAGA variance bounds. Given a direction $\mathcal G\in\mathbf{E}^*$ and step size $\eta>0$, define the Bregman proximal mapping by
\[
\operatorname{Prox}^{\psi}_{h,\eta}(x,\mathcal{G})
:=\arg\min_{u\in\mathbf{E}}
\left\{\langle \mathcal{G},u\rangle+\frac{1}{\eta}D_{\psi}(u,x)+h(u)\right\}
\]
whenever the minimizer exists. Assumption \ref{assump} ensures that this mapping is single-valued along the algorithmic trajectory.

\begin{definition}[Smooth adaptability \citep{doi:10.1137/16M1099546,bolte2018first,wang2024bregman}]
    Let $\mathcal{X}\subset\mathbf{E}$ be nonempty, convex, and open, and let $f:\mathcal{X}\to(-\infty,+\infty]$ be proper, lower semicontinuous, and continuously differentiable on $\mathcal{X}$. Given a DGF $\psi$, the pair $(f,\psi)$ is $(\bar L,\ubar L)$-smooth adaptable on $\mathcal{X}$ if constants $\bar L>0$ and $\ubar L\geq0$ satisfy, for every $x,y\in\mathcal{X}$,
    \[
        -\ubar{L}D_{\psi}(x,y)
        \leq f(x)-f(y)-\langle\nabla f(y),x-y\rangle
        \leq \bar{L}D_{\psi}(x,y).
    \]
\end{definition}
The upper inequality is the relative-smoothness bound used for descent; the lower inequality controls negative curvature. The convergence proofs below use the upper constant $\bar L$ on a deterministic region $D\subset\mathcal X$.

\begin{example}[A genuinely non-Euclidean instance]\label{ex:entropy}
Let $\mathbf E=\mathbb R^d$ carry the norm $\|\cdot\|=\|\cdot\|_1$, so that the dual norm is $\|\cdot\|_*=\|\cdot\|_\infty$, and let $h=\iota_{\Delta}$ be the indicator of the probability simplex $\Delta=\{x\geq0:\sum_i x_i=1\}$. The negative-entropy kernel $\psi(x)=\sum_i x_i\log x_i$ is $1$-strongly convex with respect to $\|\cdot\|_1$ by Pinsker's inequality, and its divergence is the Kullback--Leibler divergence $D_\psi(x,y)=\sum_i x_i\log(x_i/y_i)$, which is asymmetric and admits no polarization identity. For a finite sum $f=\frac1n\sum_i f_i$ with bounded Hessians on $\Delta$, the pair $(f,\psi)$ is relatively smooth; the SAGA correction $\widetilde\nabla_k$ and the accumulators $\xi_k=\sum|\langle y_{i,t},s_{i,t}\rangle|$ and $\delta_k=\sum\|y_{i,t}\|_\infty^2$ are then measured in $\ell_\infty$, and the proximal step is the KL projection onto $\Delta$, i.e.\ a softmax renormalization. Every constant in Theorems \ref{theo1}--\ref{theo3} is expressed through the $\ell_1/\ell_\infty$ pairing, with $\sigma_b\geq1$ typically strict; the Euclidean specialization $\psi=\frac12\|\cdot\|_2^2$ used in Section \ref{tests} is the degenerate case $\|\cdot\|=\|\cdot\|_*=\|\cdot\|_2$, $\sigma_b=1$. Standard Euclidean-underlying stochastic Bregman analyses, which measure gradients and variance in $\ell_2$, do not cover this instance directly.
\end{example}

\subsection{Safeguarded step-size framework}
The candidate $\widehat{\eta}^{\mathrm{BB}}_k$, its clipped value $q_k$, and the monotone update $\eta_k=\max\{\eta_{k-1},q_k\}$ are defined in \eqref{stableBB}--\eqref{clippedBB}. Because they use only records collected before $B_k$ is sampled, $\eta_k$ is $\mathcal F_k$-measurable. The update maintains the deterministic invariant
\[
    \eta_{\min}\leq \eta_k\leq \eta_{\max},
    \qquad
    \eta_k\geq\eta_{k-1}
    \quad \text{for all } k.
\]
The lower bound sets the scale in the convex rate, the upper bound absorbs relative smoothness and gradient-estimation error, and monotonicity telescopes the Bregman distance in the convex proof. The nonconvex argument uses only the upper bound.

\begin{remark}[Cost of monotonicity]\label{rem:monotone}
The nondecreasing rule $\eta_k=\max\{\eta_{k-1},q_k\}$ is precisely what makes $a_k=1/\eta_k$ nonincreasing and thus lets the Bregman distance telescope in Theorem \ref{theo1}. Its price is that the step may rise but never fall: starting from $\eta_{\min}$, it ramps toward the clipped curvature estimate and then holds. Overshoot is controlled not by retreat but by the fixed cap $\eta_{\max}$, so the adaptation is confined to how fast and how far the step climbs within $[\eta_{\min},\eta_{\max}]$. The nonconvex analysis (Theorem \ref{theo3}) does not use monotonicity and relies only on the upper bound, so there the step is free to vary in both directions inside the interval.
\end{remark}

\subsection{Assumptions}
\begin{assumption} \label{assump}
    Items 1--5 apply to all results; Item 6 is used only in the convex analysis.
    \begin{enumerate}
        \item \textbf{Geometry and prox.} The function $\psi$ is continuously differentiable and 1-strongly convex with respect to $\|\cdot\|$. The function $h:\mathbf{E}\to(-\infty,+\infty]$ is proper, closed, and convex. For every $x\in D$, $\mathcal{G}\in\mathbf{E}^*$, and $\eta\in[\eta_{\min},\eta_{\max}]$, the Bregman proximal subproblem has a unique solution.
        \item \textbf{Deterministic bounded region.} There exists a deterministic bounded convex set $D\subset\mathcal{X}$ such that $x_k,\phi_i^k\in D$ almost surely for all $k$ and $i$. In the convex results, an optimal solution $x^*$ also belongs to $D$.
        \item \textbf{Sampling.} Conditional on $\mathcal{F}_k$, the mini-batch is sampled uniformly from $[n]$ with size $b$ and without replacement. Write $I_1,\ldots,I_b$ for the sampled indices. A deterministic mini-batch second-moment constant $\sigma_b\geq 1$ satisfies, for any deterministic $v_1,\ldots,v_n\in\mathbf{E}^*$ with $\bar v=\frac1n\sum_i v_i$,
        \[
            \mathbb{E}\left[
            \left\|
            \frac1b\sum_{j=1}^b (v_{I_j}-\bar v)
            \right\|_*^2
            \right]
            \leq
            \frac{\sigma_b}{bn}\sum_{i=1}^n\|v_i-\bar v\|_*^2.
        \]
        In the Euclidean norm, one may take $\sigma_b=1$. Conditional on $\mathcal F_k$, the table vectors are fixed, so this inequality applies directly to the SAGA estimator.
        \item \textbf{Relative smoothness.} The pair $(f,\psi)$ is $(\bar L,\ubar L)$-smooth adaptable on $D$; only the upper bound with $\bar L$ is used in the convergence proofs. This is a relative smoothness condition for the aggregate differentiable term $f$ with respect to the chosen kernel $\psi$.
        \item \textbf{Component local smoothness.} Each $\nabla f_i$ is Lipschitz on $D$ with deterministic constant $M$, i.e.,
        \[
            \|\nabla f_i(x)-\nabla f_i(y)\|_*\leq M\|x-y\|,
            \qquad x,y\in D,\ i\in[n].
        \]
        \item \textbf{Convex self-bounding condition.} For the convex results, each $f_i$ is convex on $D$ and satisfies
        \[
            \|\nabla f_i(x)-\nabla f_i(y)\|_*^2
            \leq
            2M\left(f_i(x)-f_i(y)-\langle\nabla f_i(y),x-y\rangle\right),
            \qquad x,y\in D.
        \]
    \end{enumerate}
\end{assumption}
\paragraph{Role of the assumptions.}
The deterministic set $D$ localizes the component smoothness needed for variance reduction. A concrete verifiable sufficient condition is that $\operatorname{dom}h$ is compact and the kernel $\psi$ is continuous on $\operatorname{dom}h$: the Bregman proximal subproblem then minimizes a lower-semicontinuous function over a compact feasible set, so every iterate $x_k$ and every table point $\phi_i^k$ stays in $D:=\operatorname{dom}h$ almost surely, and $M$ is the finite Lipschitz constant of the component gradients on that compact set. This makes Item 2 an assumption on the problem data rather than on the trajectory. The sparse-NMF experiment of Section \ref{tests} does not satisfy this out of the box, because the objective is invariant under the rescaling $(U,V)\mapsto(tU,t^{-1}V)$ and its sublevel sets are therefore unbounded; appending an explicit box $0\le U,V\le \mathcal U$ to the cardinality and nonnegativity constraints restores a compact $D$ while leaving the optimal value unchanged for $\mathcal U$ large enough, at the cost of enforcing the box in the proximal step. Fixing $D$ before the run keeps $M$ deterministic. This local component condition serves a different purpose from relative smoothness: $\bar L$ controls descent of the aggregate $f$ in Bregman geometry, whereas $M$ controls differences between stored component gradients.

For convex components, Item 6 converts those gradient differences into the linearization errors used by the objective-gap Lyapunov function. Convex Euclidean $M$-smooth functions satisfy this self-bounding inequality automatically. The nonconvex analysis instead tracks distances between the current iterate and the SAGA table points, and therefore uses only Items 1--5.

\subsection{Variance bounds}
For the convex analysis, let $x^*$ be an optimal solution in $D$ and choose $\rho^*\in\partial h(x^*)$ satisfying the first-order condition
\[
    \nabla f(x^*)+\rho^*=0.
\]
Define
\[
    \Delta_i(x):=f_i(x)-f_i(x^*)-\langle\nabla f_i(x^*),x-x^*\rangle,
    \qquad
    H(x):=h(x)-h(x^*)-\langle\rho^*,x-x^*\rangle.
\]
Here $\Delta_i(x)$ is the linearization error of component $f_i$, $H(x)$ is the corresponding error of $h$, and $\Delta(x):=\frac1n\sum_i\Delta_i(x)$. Convexity makes all three quantities nonnegative; the optimality condition also gives $F(x)-F(x^*)=\Delta(x)+H(x)$. This identity links the SAGA variance to the objective gap.

\begin{lemma}[Convex SAGA variance] \label{VB}
    Under Assumption \ref{assump}, the SAGA estimator in Algorithm \ref{alg:stc_schm} is conditionally unbiased,
    \[
        \mathbb{E}_k[\widetilde{\nabla}_k]=\nabla f(x_k)
    \]
    and its conditional variance in the convex setting satisfies
    \[
        \mathbb{E}_k\left[
        \|\widetilde{\nabla}_k-\nabla f(x_k)\|_*^2
        \right]
        \leq
        \frac{16\sigma_bM}{b}
        \left[
        \Delta(x_k)+\frac1n\sum_{i=1}^n\Delta_i(\phi_i^k)
        \right].
    \]
\end{lemma}

\subsection{Convex analysis}
Lemma \ref{VB} makes the stochastic error proportional to the current and stored objective gaps. Combining this estimate with the three-point Bregman proximal inequality yields a Lyapunov decrease in which the safeguard absorbs the variance term.
\begin{theorem} \label{theo1}
    Suppose Assumption \ref{assump} holds in the convex setting, and the safeguarded step-sizes satisfy
    \[
        \eta_{\min}\leq \eta_k\leq\eta_{\max},\qquad
        \eta_k\geq \eta_{k-1},\qquad
        \eta_{\max}\leq \frac{1}{\bar L+32\sigma_bM/b}.
    \]
    Let $\bar{x}_{K}=\frac1K\sum_{k=1}^{K}x_k$. Then
    \[
        \mathbb{E}[F(\bar{x}_{K})]-F(x^*)
        \leq
        \frac{2}{K}
        \left[
            \left(\frac{n}{4b}+\frac12\right)(F(x_0)-F(x^*))
            +\frac{1}{\eta_{\min}}D_\psi(x^*,x_0)
        \right].
    \]
\end{theorem}
Theorem \ref{theo1} therefore gives an $O(n/K)$ ergodic rate for fixed $b$. Its dependence on $\eta_{\min}$ records the cost of the lower safeguard, while the admissible $\eta_{\max}$ couples relative smoothness and SAGA variance. The BB candidate remains free to adapt anywhere inside this interval.

\begin{remark}[What the guarantee does and does not claim]\label{rem:guarantee}
The admissible cap $\eta_{\max}\le(\bar L+32\sigma_bM/b)^{-1}$ encodes exactly the smoothness constants $\bar L$ and $M$ that a well-tuned fixed step would use. Consequently Theorem \ref{theo1} does not improve on the $O(n/K)$ rate of fixed-step BPSG-SAGA; it certifies that adapting within $[\eta_{\min},\eta_{\max}]$ does not degrade that rate. Ada-BPSG is thus adaptive within a safe interval whose upper end still depends on problem constants, not free of them; the empirical payoff (Section \ref{tests}) is robustness to the \emph{initial} step size rather than a faster worst-case rate. The leading constant is moreover governed by $\eta_{\min}^{-1}D_\psi(x^*,x_0)$, so a conservative $\eta_{\min}$ enlarges the bound: the BB candidate removes the need to guess \emph{where} in $[\eta_{\min},\eta_{\max}]$ the step should sit, but not the need to set the interval from $\bar L$ and $M$. The analysis is nonetheless conducted entirely in the general normed space $\mathbf{E}$: it uses only the dual norm, dual-norm Young inequalities, and the three-point identity of Lemma \ref{lem:prox-three-point}, and at no point invokes the Euclidean polarization identity $2\langle a,b\rangle=\|a\|^2+\|b\|^2-\|a-b\|^2$ or the norm decomposability on which the Euclidean adaptive-BB analyses rely. The price of dropping them is explicit rather than hidden: the mini-batch second-moment constant obeys $\sigma_b\geq1$ with equality only in the Euclidean norm (Assumption \ref{assump}, Item 3), and the variance bound of Lemma \ref{VB} carries a factor $4$ where the Hilbert variance identity $\sum_i\|a_i-\bar a\|^2=\sum_i\|a_i\|^2-n\|\bar a\|^2$ would give $1$. Recovering the Euclidean $O(n/K)$ order in this non-Hilbert geometry, rather than improving it, is where the difficulty lies, so the contribution is added generality at no loss of rate.
\end{remark}

Relative quadratic growth means that the objective gap controls the Bregman distance to a solution. Under this condition, restarting the averaged convex guarantee turns a constant-factor reduction per round into geometric convergence.
\begin{theorem} \label{theo2}
    In addition to the assumptions of Theorem \ref{theo1}, suppose $F$ satisfies the relative quadratic growth condition
    \[
        F(x)-F(x^*)\geq \mu D_\psi(x^*,x),
        \qquad x\in D.
    \]
    Let
    \[
        A_b:=\frac{n}{4b}+\frac12.
    \]
    If each restart round reinitializes the SAGA table at $x^{(t)}$, starts or continues with a step-size in $[\eta_{\min},\eta_{\max}]$ satisfying the monotonicity condition within the round, and runs Theorem \ref{theo1} for
    \[
        K\geq 4\left(A_b+\frac{1}{\mu\eta_{\min}}\right)
    \]
    iterations and sets $x^{(t+1)}$ to the averaged iterate of round $t$, then after $\tau$ restart rounds,
    \[
        \mathbb{E}[F(x^{(\tau)})]-F(x^*)
        \leq
        2^{-\tau}\left(F(x^{(0)})-F(x^*)\right).
    \]
\end{theorem}
The same condition converts the objective contraction into a Bregman distance bound:
\[
    \mathbb{E}\bigl[D_\psi(x^*,x^{(\tau)})\bigr]
    \leq
    \frac{2^{-\tau}}{\mu}
    \left(F(x^{(0)})-F(x^*)\right).
\]
If the kernel is 1-strongly convex, this further bounds the squared norm distance by
$\mathbb{E}\|x^{(\tau)}-x^*\|^2\leq 2\mathbb{E}[D_\psi(x^*,x^{(\tau)})]$ whenever the displayed Bregman distance is evaluated on $D$.

\subsection{Nonconvex analysis}
We next allow $f$ to be nonconvex and assume that $F$ has a finite lower bound $F_{\inf}$. Because an objective gap alone does not characterize stationarity, we use a Bregman proximal gradient mapping.
\begin{definition}[Bregman proximal gradient mapping]
For $\eta>0$, define
\[
    \mathcal{G}_{\eta}(x)
    =
    \frac{x-\operatorname{Prox}^{\psi}_{h,\eta}(x,\nabla f(x))}{\eta}.
\]
\end{definition}
Let
\[
    \widetilde{x}_{k+1}
    =
    \operatorname{Prox}^{\psi}_{h,\eta_k}(x_k,\nabla f(x_k)).
\]
The full-gradient proximal point $\widetilde{x}_{k+1}$ removes the SAGA estimation error from the stationarity measure. We normalize its displacement by the fixed upper safeguard:
\[
    \mathcal{R}_k:=\frac{x_k-\widetilde{x}_{k+1}}{\eta_{\max}}.
\]
This normalization yields a common scale across adaptive iterations. The residual $\mathcal R_k$ vanishes exactly when the full-gradient Bregman proximal step leaves $x_k$ unchanged.

\begin{lemma}[Nonconvex SAGA variance] \label{nonconvexVariance}
    Under Items 1--5 of Assumption \ref{assump},
    \[
        \mathbb{E}_k\left[
        \|\widetilde{\nabla}_k-\nabla f(x_k)\|_*^2
        \right]
        \leq
        \frac{4\sigma_bM^2}{bn}
        \sum_{i=1}^{n}\|x_k-\phi_i^k\|^2.
    \]
\end{lemma}

\begin{theorem} \label{theo3}
    Suppose the common conditions in Assumption \ref{assump} hold, $f$ may be nonconvex, and $F$ is bounded below by $F_{\inf}>-\infty$. Let
    \[
        \Gamma_b:=
        \frac{8\sigma_bM^2n(2n-b)}{b^3}.
    \]
    If
    \[
        \Gamma_b\eta_{\max}+\bar L
        \leq
        \frac{1}{\eta_{\max}},
    \]
    then for $R$ drawn independently and uniformly from $\{0,\ldots,K-1\}$,
    \[
        \mathbb{E}\left[\|\mathcal{R}_R\|^2\right]
        \leq
        \frac{2(F(x_0)-F_{\inf})}{K\eta_{\max}}.
    \]
\end{theorem}
Theorem \ref{theo3} gives an $O(1/K)$ bound on the expected squared residual. The quantity $\Gamma_b$ captures how mini-batch sampling and table refresh interact with local component smoothness, while the upper safeguard balances this variance term against relative smoothness. Because $\Gamma_b=O(M^2n^2/b^3)$, the stability condition $\Gamma_b\eta_{\max}+\bar L\le\eta_{\max}^{-1}$ forces $\eta_{\max}=O(b^{3/2}/(nM))$, so the residual bound $2(F(x_0)-F_{\inf})/(K\eta_{\max})$ inherits the $O(n)$-type gradient complexity characteristic of SAGA rather than the improved $O(\sqrt n)$ dependence available to SARAH- or SPIDER-type recursive estimators. We nonetheless build on SAGA because its table supplies the component secant pairs $(s_{i,t},y_{i,t})$ directly, which is what the stabilized candidate consumes; porting the same candidate to a memory-light recursive estimator is left to future work. Section \ref{tests} examines how the stabilized candidate behaves with practically chosen safeguards and kernels.

\section{Numerical Experiments} \label{tests}
The experiments probe the mechanisms developed above: whether the adaptive SAGA step improves gradient efficiency, whether stabilization reduces sensitivity to the initial scale, whether the convex $O(n/K)$ guarantee is realized inside a genuinely non-Euclidean instance that meets every hypothesis of Assumption \ref{assump}, whether the same behavior transfers from synthetic to real measured data, and whether the design remains effective on a harder non-Euclidean problem that falls outside the theory. We use $\mathbf E=\mathbb R^d$ throughout and vary the distance-generating function $\psi$ with the problem geometry.

\subsection{Regularized binary classification} \label{exp-saga-bb}
We first choose the Euclidean kernel $\psi(x)=\frac12\|x\|_2^2$ and solve $\ell_2$-regularized logistic regression,
\[
    F(x)=\frac1n\sum_{i=1}^n
    \log(1+\exp(-b_i a_i^Tx))
    +\frac{\lambda}{2}\|x\|_2^2,
    \qquad \lambda=\frac1n,
\]
where $a_i\in\mathbb R^d$ and $b_i\in\{-1,+1\}$ are the features and label of sample $i$. Under this kernel, Ada-BPSG specializes to SAGA-BB. We compare it with SAGA, SVRG, SVRG-BB, loopless SVRG, and SARAH on \emph{mushrooms}, \emph{ijcnn1}, \emph{w8a}, and \emph{covtype} from LibSVM \citep{chang2011libsvm}. Every method uses the same tuning grid: the initial step size lies in $\{10^{-3},10^{-2},10^{-1},1,10,100\}$ and the batch size lies in $\{1,8,16,64\}$. For methods with an epoch length or refresh interval, we set it to $m=n/b$.

\begin{figure*}[t!]
    \centering
    \begin{subfigure}[b]{0.235\textwidth}
        \centering
        \includegraphics[width=\linewidth]{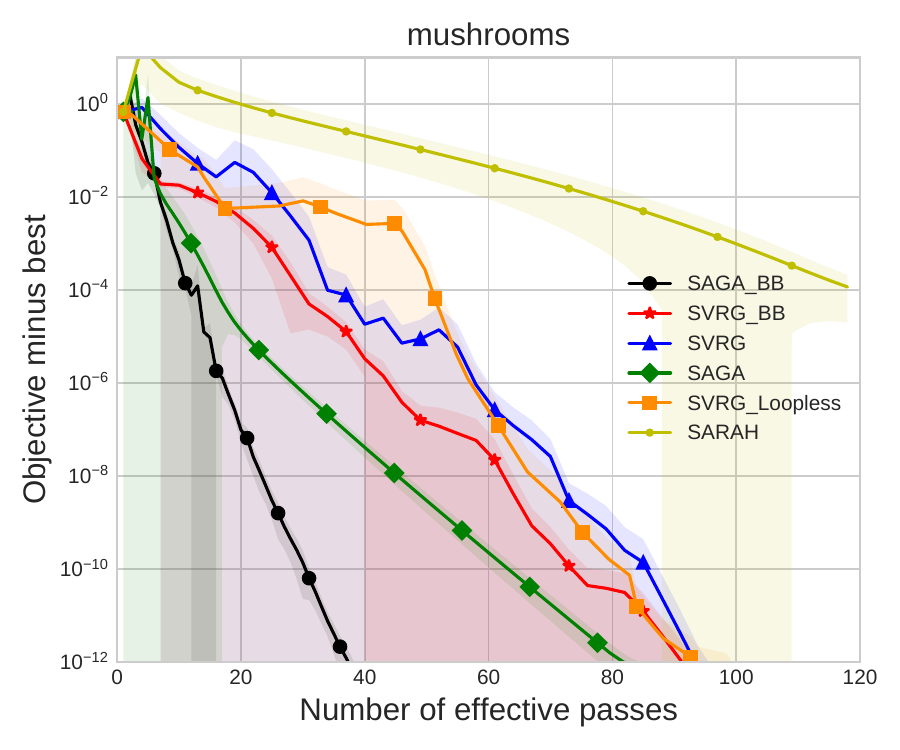}
        \caption{}
    \end{subfigure}\hfill
    \begin{subfigure}[b]{0.235\textwidth}
        \centering
        \includegraphics[width=\linewidth]{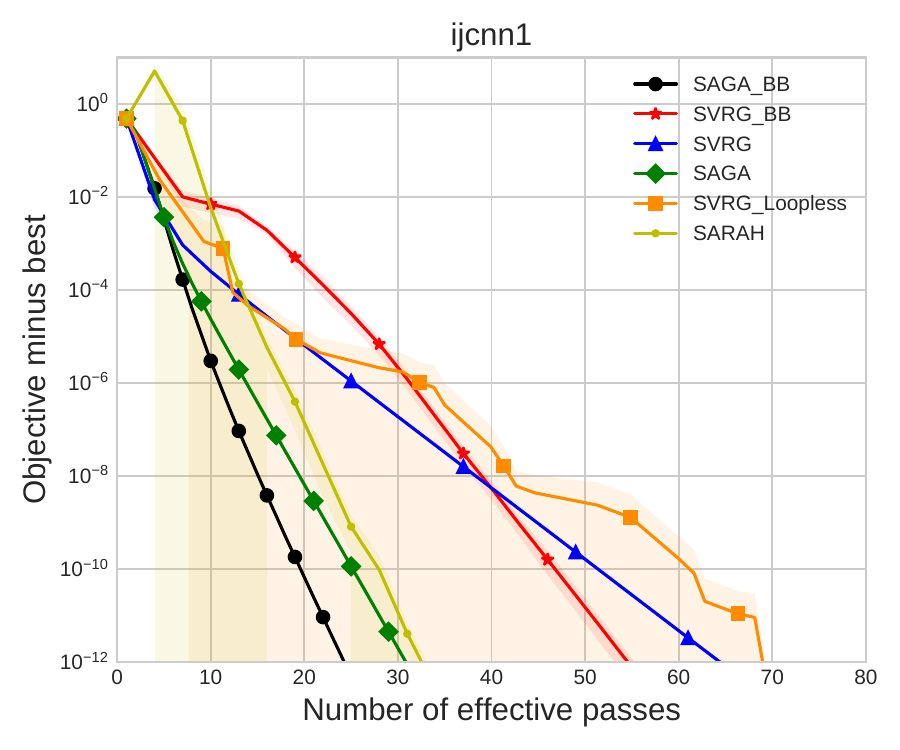}
        \caption{}
    \end{subfigure}\hfill
    \begin{subfigure}[b]{0.235\textwidth}
        \centering
        \includegraphics[width=\linewidth]{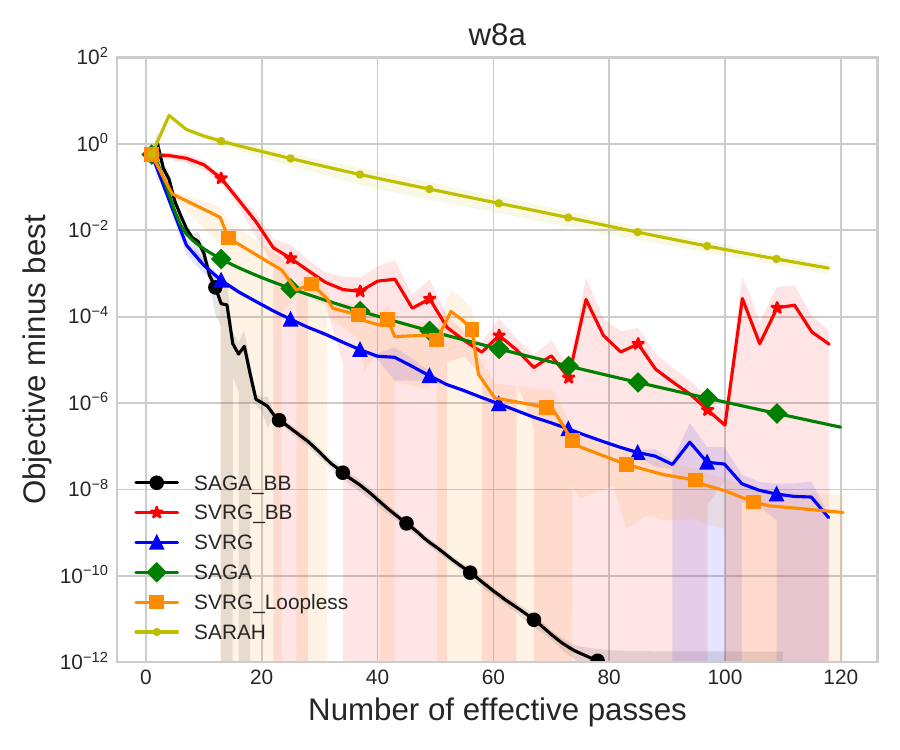}
        \caption{}
    \end{subfigure}\hfill
    \begin{subfigure}[b]{0.235\textwidth}
        \centering
        \includegraphics[width=\linewidth]{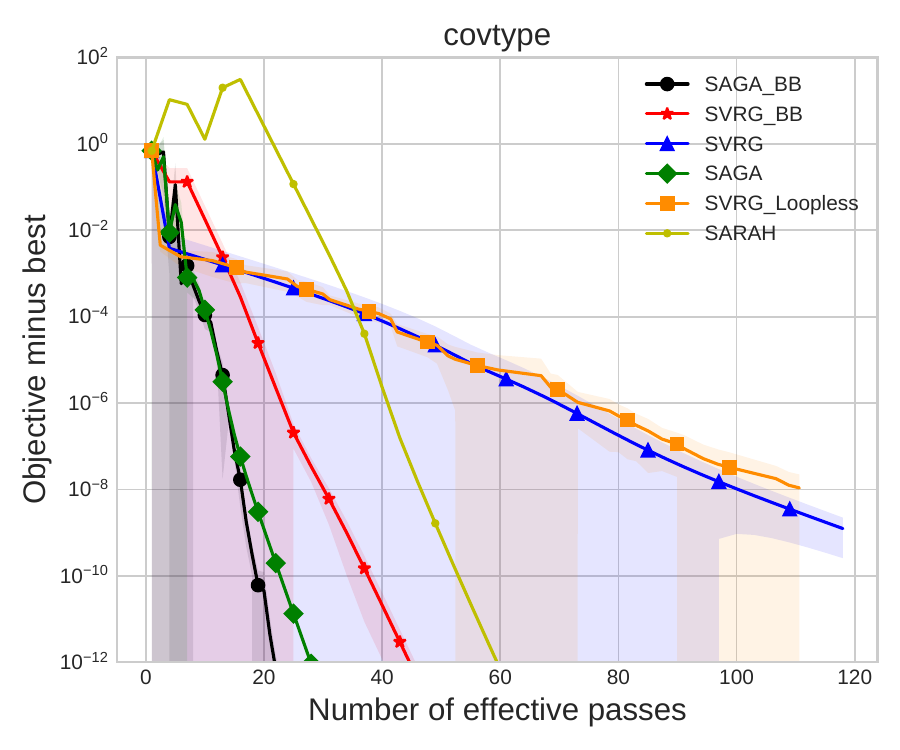}
        \caption{}
    \end{subfigure}
    \caption{Gradient efficiency for regularized logistic regression. The horizontal axis counts component-gradient evaluations normalized by $n$, and the vertical axis is the objective gap to the best value attained in the corresponding experiment. Lines and shaded regions show the mean and standard deviation over five runs. All methods use their best configuration from the common tuning grid.}
    \label{fig:loss}
\end{figure*}

Figure \ref{fig:loss} connects the curvature diagnostic in Figure \ref{fig: local information} to optimization performance. SAGA-BB reaches the smallest plotted objective gaps with the fewest effective passes on all four datasets. The gain is especially pronounced on \emph{mushrooms}, \emph{w8a}, and \emph{covtype}, where its curve separates early from both SAGA and the SVRG-family baselines. Thus the stabilized secant information does more than prevent isolated step-size spikes: it improves the gradient efficiency of the underlying SAGA update.

\begin{figure*}[t!]
    \centering
    \begin{subfigure}[b]{0.235\textwidth}
        \centering
        \includegraphics[width=\linewidth]{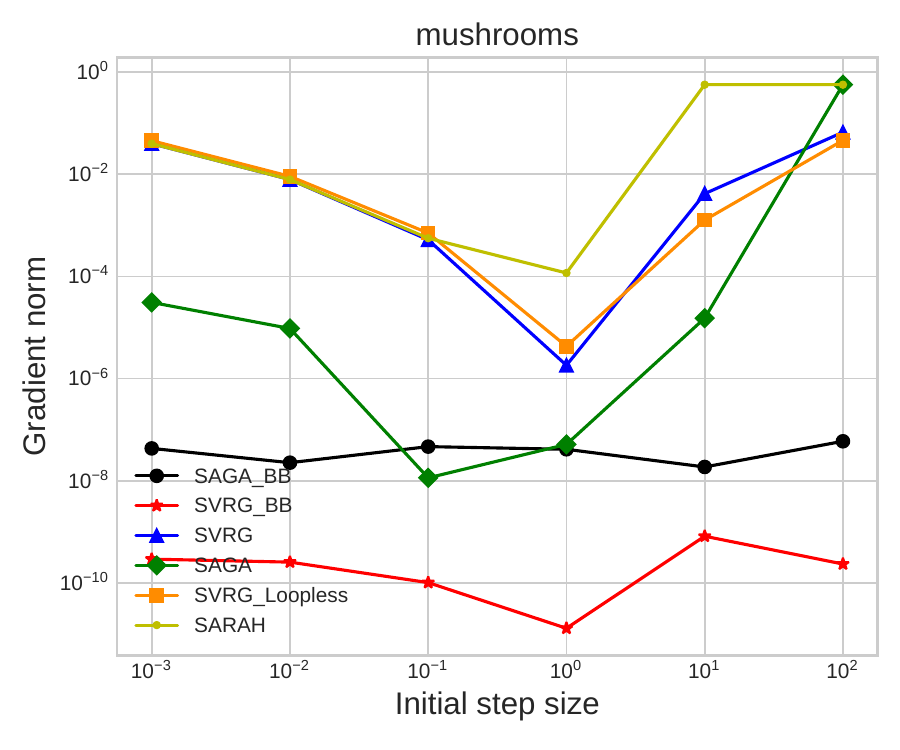}
        \caption{}
    \end{subfigure}\hfill
    \begin{subfigure}[b]{0.235\textwidth}
        \centering
        \includegraphics[width=\linewidth]{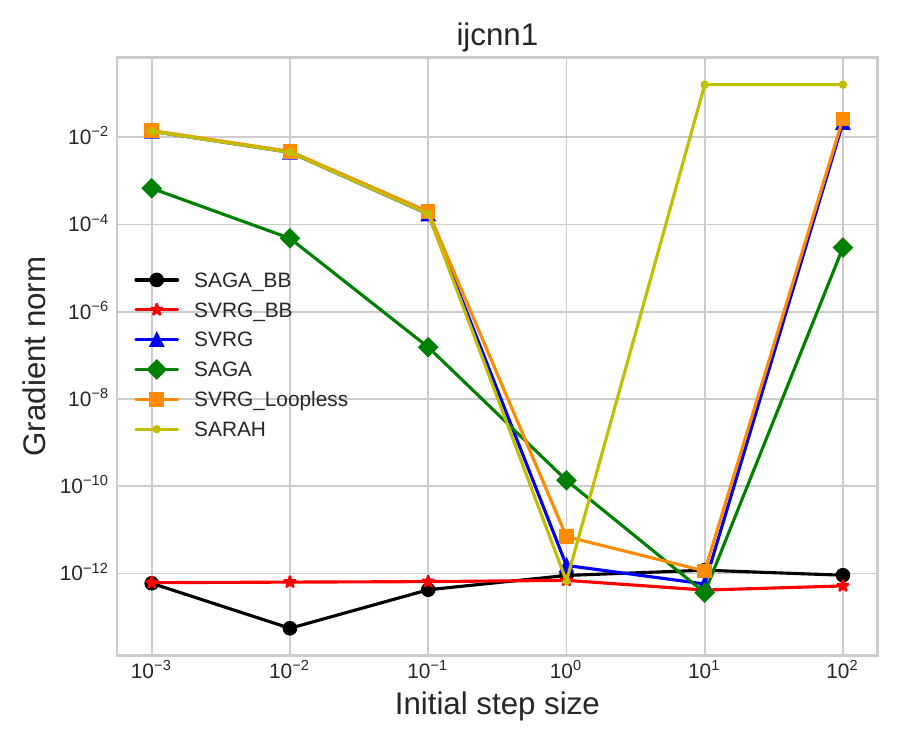}
        \caption{}
    \end{subfigure}\hfill
    \begin{subfigure}[b]{0.235\textwidth}
        \centering
        \includegraphics[width=\linewidth]{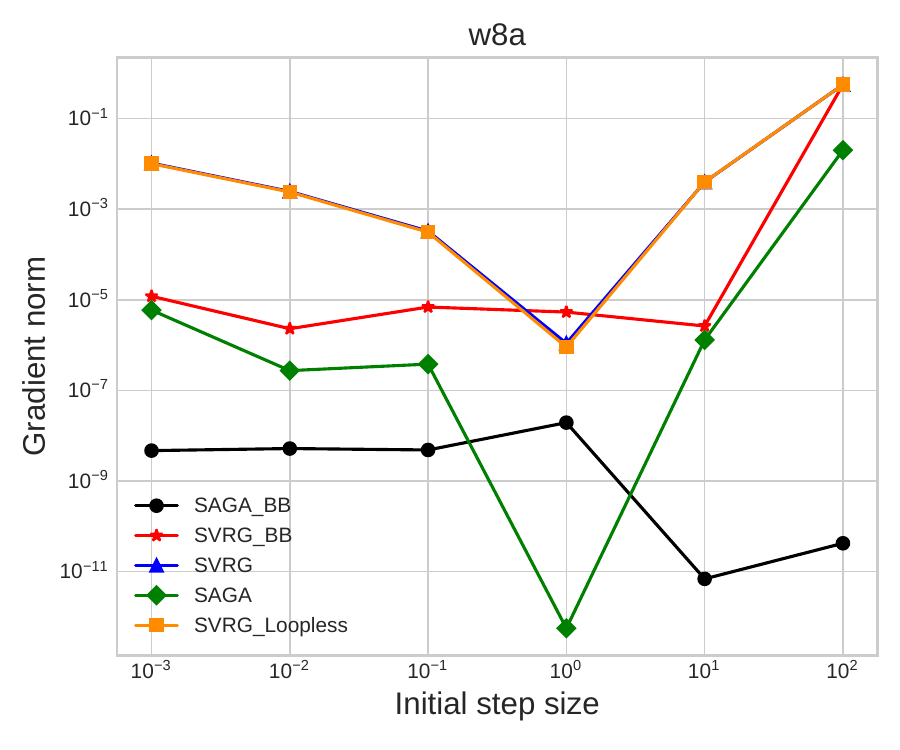}
        \caption{}
    \end{subfigure}\hfill
    \begin{subfigure}[b]{0.235\textwidth}
        \centering
        \includegraphics[width=\linewidth]{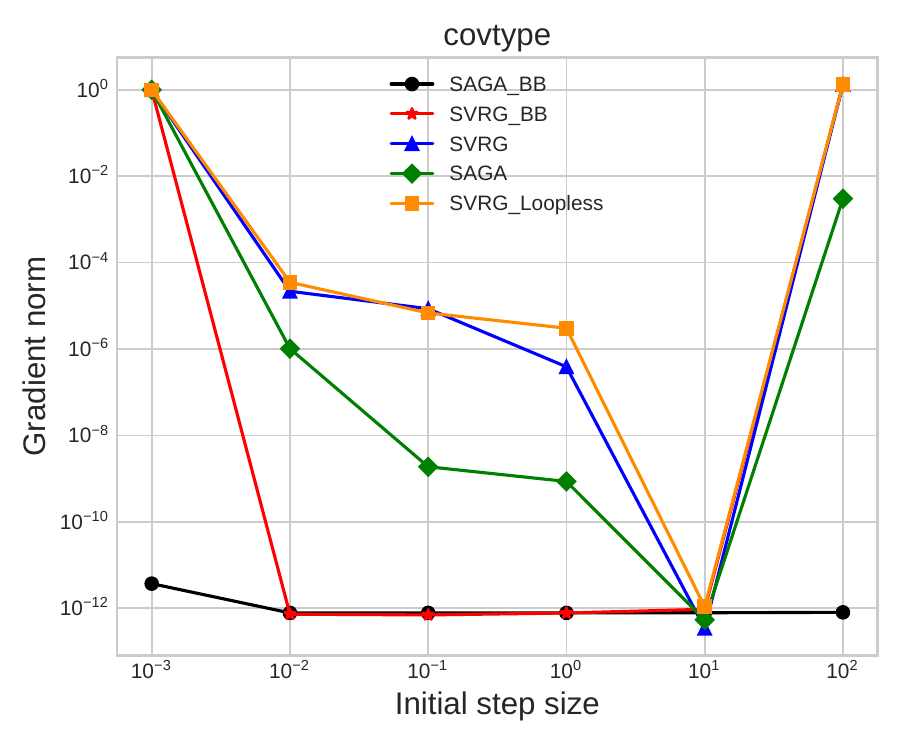}
        \caption{}
    \end{subfigure}
    \caption{Sensitivity to the initial step size after 120 effective passes with batch size $b=8$. Each part reports the final gradient norm; finite runs within the displayed range are shown.}
    \label{fig:grad_norm-step_size}
\end{figure*}

The initial-scale experiment in Figure \ref{fig:grad_norm-step_size} fixes $b=8$ and measures the final gradient norm after 120 effective passes. The vanilla VR methods vary by several orders of magnitude across the grid and can become unstable at its extremes. SVRG-BB reduces this dependence on some datasets but remains uneven on \emph{w8a}. SAGA-BB maintains a low gradient norm across the entire grid on every dataset. This behavior matches the algorithmic design: the initial value governs only the iterations before the first refresh, after which the mediant candidate adapts within the safeguard interval. Appendix \ref{Additional experiments} verifies the same pattern across additional batch sizes and with Huber loss.

\subsection{Simplex-constrained Poisson inverse problem} \label{exp:simplex}
The second experiment moves to a non-Euclidean kernel while staying strictly inside the theory, and is chosen so that variance reduction is genuinely needed. We recover mixing weights on the probability simplex from count observations,
\begin{equation} \label{simplex-poisson}
    \min_{x\in\Delta}\ F(x)=\frac1n\sum_{i=1}^n
    \Big(\langle a_i,x\rangle-b_i\log\langle a_i,x\rangle\Big),
    \qquad
    \Delta=\Big\{x\in\mathbb R^d:\ x\geq0,\ \textstyle\sum_j x_j=1\Big\},
\end{equation}
where the rows $a_i>0$ are known nonnegative dictionary atoms and $b_i\sim\mathrm{Poisson}(\langle a_i,x^\star\rangle)$ are counts generated by a planted simplex vector $x^\star$. Each $f_i(x)=\Psi_i(\langle a_i,x\rangle)$ with $\Psi_i(u)=u-b_i\log u$ is a convex linear-prediction loss, so the $O(n)$ scalar-table implementation of Appendix \ref{Efficient Implementation} applies verbatim, and $h=\iota_\Delta$ is the indicator of the simplex. Following Example \ref{ex:entropy} we equip $\mathbf E=\mathbb R^d$ with the $\ell_1$ norm, its dual $\ell_\infty$, and the negative-entropy kernel $\psi(x)=\sum_j x_j\log x_j$, which is $1$-strongly convex with respect to $\|\cdot\|_1$; the Bregman proximal step of Algorithm \ref{alg:stc_schm} then has the closed form of an entropic mirror update,
\[
    x_{k+1,j}
    =\frac{x_{k,j}\,\exp(-\eta_k[\widetilde\nabla_k]_j)}
    {\sum_{l} x_{k,l}\,\exp(-\eta_k[\widetilde\nabla_k]_l)},
\]
and the accumulators $\xi_k,\delta_k$ are formed in the $\ell_\infty$ dual norm. Because the atoms are strictly positive, the prediction obeys $\langle a_i,x\rangle\geq\underline p:=\min_i\min_j a_{ij}>0$ on the simplex, which keeps the loss away from the singularity of $\log$ and makes every constant in Assumption \ref{assump} finite: $\Delta$ is compact and convex, so the deterministic region $D=\Delta$ is available a priori; each $\Psi_i$ is convex with $\Psi_i''(u)=b_i/u^2$, giving the component-gradient bound $M=(\max_i b_i/\underline p^{\,2})\max_i\|a_i\|_\infty^2$ of Items 5--6 and a finite relative-smoothness scale $\bar L$ against the entropy kernel. Crucially, unlike a quadratic data term, the Poisson loss has \emph{no} global Euclidean Lipschitz gradient, so constant-step SGD retains a variance floor while a variance-reduced step can drive the gap arbitrarily low: this is exactly the regime the method targets, and the convex rate of Theorem \ref{theo1} is expected to hold in it. The geometry is non-Hilbert throughout, with $\sigma_b>1$ in general.

We use a controlled synthetic instance with $d=500$, $n=5000$, atoms $a_i$ with entries drawn uniformly in $[0.05,1.05]$, a planted interior signal $x^\star$, and Poisson counts $b_i$; the batch size is $b=8$ and the refresh interval is $m=n/b$. Because the problem is convex, we obtain the reference value $F^\star$ from an independent high-accuracy full-batch entropic mirror descent with backtracking, so that no plotted stochastic method is scored against its own best iterate. We compare Ada-BPSG against deterministic BPG, against BPSG with SGD, SAGA, and SARAH estimators, and against the extrapolated line-search variant BPSGE. The decisive feature of this instance is that its worst-case component constant $M$ is dominated by the $\log$ singularity and is therefore far larger than the curvature the iterates actually visit: here $\bar L\approx18$ but $M\approx2.2\times10^3$, whereas the local component smoothness at typical predictions is only $M_{\mathrm{loc}}\approx18$. Figure \ref{fig:simplex} (a)--(b) therefore adopt an \emph{a-priori} protocol: every fixed-step baseline uses the largest step its own analysis guarantees from these quantities---$(\bar L+M)^{-1}$ for SGD, $(\bar L+32M/b)^{-1}$ for SAGA, and $(\bar L+M/b)^{-1}$ for the recursive SARAH estimator---while Ada-BPSG adapts within a safeguard $\eta_{\max}=(\bar L+M_{\mathrm{loc}}/b)^{-1}$ built from the local constant, started from the lower safeguard $\eta_{\min}$. This local cap exceeds the worst-case safeguard $(\bar L+32\sigma_b M/b)^{-1}$ of Theorems \ref{theo1}--\ref{theo3}, so it is a heuristic relaxation rather than a theory-faithful setting; to make the gap explicit we also plot ``Ada-BPSG (theory cap)'', the same rule confined to the analysis cap with the global $M$. The complementary question---what a baseline achieves if its step is instead hand-tuned per instance---is answered separately in Figure \ref{fig:simplex} (c).

\begin{figure*}[t!]
    \centering
    \begin{subfigure}[b]{0.31\textwidth}
        \centering
        \includegraphics[width=\linewidth]{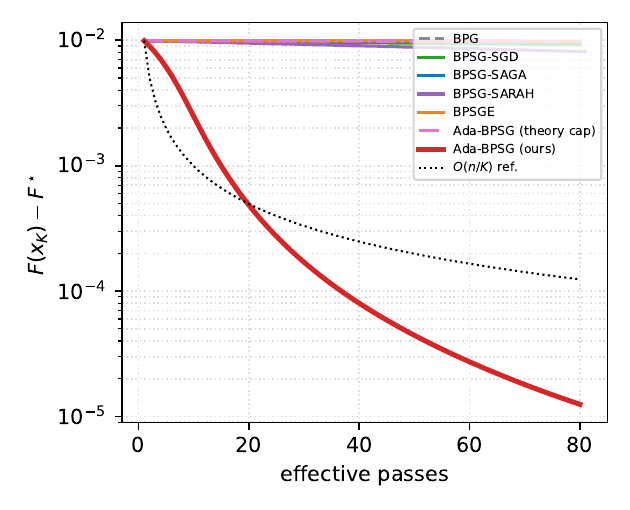}
        \caption{objective gap vs.\ effective passes}
    \end{subfigure}\hfill
    \begin{subfigure}[b]{0.31\textwidth}
        \centering
        \includegraphics[width=\linewidth]{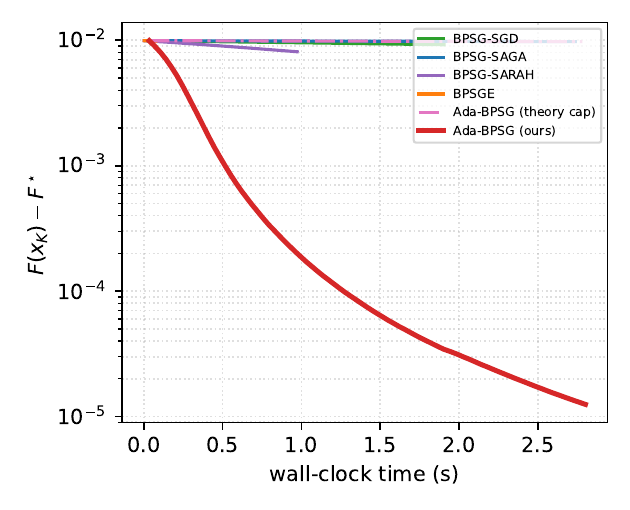}
        \caption{objective gap vs.\ wall-clock time}
    \end{subfigure}\hfill
    \begin{subfigure}[b]{0.31\textwidth}
        \centering
        \includegraphics[width=\linewidth]{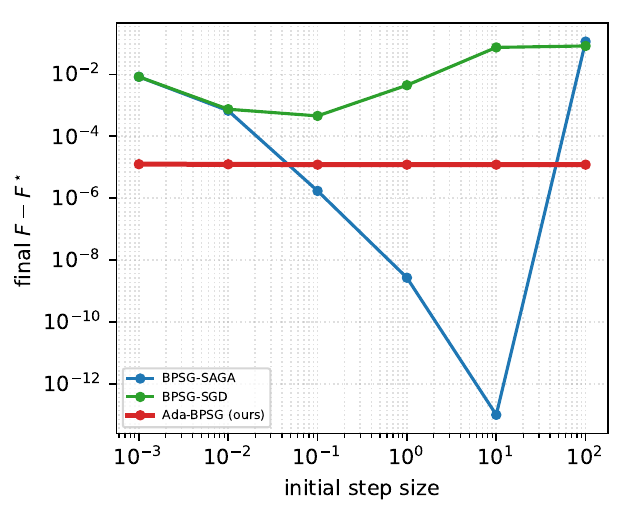}
        \caption{sensitivity to the initial step}
    \end{subfigure}
    \caption{Simplex-constrained Poisson inverse problem under the entropy kernel. (a) Objective gap $F(x_K)-F^\star$ against effective passes, with the $O(n/K)$ reference slope of Theorem \ref{theo1} overlaid; each fixed-step baseline uses the a-priori step guaranteed by its own analysis, and ``Ada-BPSG (theory cap)'' is the adaptive rule confined to the worst-case safeguard, whereas ``Ada-BPSG (ours)'' relaxes the cap to the local component constant. (b) The same gap against wall-clock time. (c) Final objective gap after a fixed budget as the \emph{initial} step sweeps the grid $\{10^{-3},\dots,10^{2}\}$: for the fixed-step methods the swept value is the step, while for Ada-BPSG it is only the starting point of the adaptive rule. Lines and shaded regions show the mean and standard deviation over five runs at batch size $b=8$.}
    \label{fig:simplex}
\end{figure*}

Figure \ref{fig:simplex} reads along the same three axes as the classification study, now in Bregman geometry. In Figure \ref{fig:simplex} (a) the gap of Ada-BPSG decays steadily to $1.3\times10^{-5}$ along the $O(n/K)$ reference slope, exhibiting the convex-rate behavior predicted by Theorem \ref{theo1} in a bona fide non-Euclidean instance. The fixed-step baselines, held to the steps their own theory licenses, remain near $10^{-2}$: because the worst-case $M$ is inflated by the $\log$ singularity, the guaranteed steps are two to four orders of magnitude below the useful curvature scale, and the methods barely move in $80$ passes. The ``theory cap'' run of Ada-BPSG, confined to the same worst-case safeguard, sits with them---so on this instance the analysis-faithful step is itself conservative, and the speedup comes precisely from adapting \emph{beyond} it. This is the intended behavior of the stabilized rule: its secant candidate climbs from $\eta_{\min}$ to the local safeguard, improving on the best a-priori baseline by more than two orders of magnitude while reading only the SAGA information it already maintains, and it does so as a heuristic relaxation of the worst-case cap, as is standard for BB-type steps. Figure \ref{fig:simplex} (b) shows the same separation in wall-clock time: Ada-BPSG reaches $10^{-5}$ in a few seconds, each step-size refresh costing only two scalar accumulations and no line search. Figure \ref{fig:simplex} (c) explains what tuning would buy the baselines and why adaptivity still wins in practice: the final gap of Ada-BPSG is essentially flat across the entire initial-step grid, whereas fixed-step SAGA is sharply V-shaped---it can reach a much lower gap, but only in a narrow band around a single step several orders of magnitude above the guaranteed cap, and it diverges beyond it---while SGD is pinned at its variance floor. Ada-BPSG thus recovers the accuracy of a hand-tuned step without any tuning and without the attendant fragility. Together Figure \ref{fig:simplex} exhibits the predicted $O(n/K)$ decay and the practical value of adapting past the worst-case safeguard, complementing the harder out-of-theory test that follows.

\subsection{Real-data confirmation: hyperspectral unmixing} \label{exp:hsi}
The synthetic study isolates the mechanism on controlled data; we now verify that it survives on measured observations, with the same estimator, safeguards, and hyperparameters. We use the Samson scene \citep{zhu2017hyperspectral}, $95\times95$ pixels over $L=156$ contiguous bands with three reference materials (soil, tree, water). To assemble a large finite sum on the simplex, we pool a spatially contiguous, spectrally homogeneous patch $R$ under a shared-abundance model: all pixels in $R$ are taken to share one abundance $x\in\Delta_d$ over a dictionary $D\in\mathbb R_+^{L\times d}$ of measured spectra. Each material contributes an \emph{endmember bundle} of $K=6$ representative spectra---cluster centers of the pixels it dominates, the standard construction for spectral variability \citep{somers2012endmember}---so $d=18$, and summing the weights within a bundle recovers a physical material abundance. Every pixel--band pair $(p,\ell)$ with $p\in R$ is one component, with atom $a_i=D_{\ell,:}\ge0$ and count $b_i=Y_{p\ell}\ge0$ the measured radiance, giving
\[
    \min_{x\in\Delta}\ \frac1n\sum_{i=1}^{n}
    \Big(\langle a_i,x\rangle-b_i\log\langle a_i,x\rangle\Big),
    \qquad n=|R|\cdot L,
\]
the same generalized Kullback--Leibler (Poisson) fidelity as \eqref{simplex-poisson}, now with real atoms and radiances. We select the $10\times10$ patch that is simultaneously most mixed and most homogeneous---mean abundance $(0.37,0.53,0.10)$, a genuinely interior target---which yields $n=15\,600$. Exactly as in the synthetic case, the worst-case component constant is inflated by measured reflectances that dip to $\underline p\approx8.6\times10^{-3}$ in some bands: here $\bar L\approx68$ and $M\approx5.6\times10^{3}$, whereas the local component smoothness is only $M_{\mathrm{loc}}\approx23$. We keep the batch size $b=8$, the a-priori baseline protocol, and the two Ada-BPSG safeguards of Section~\ref{exp:simplex} unchanged, and again obtain $F^\star$ from an independent full-batch mirror-descent solver.

\begin{figure*}[t!]
    \centering
    \begin{subfigure}[b]{0.31\textwidth}
        \centering
        \includegraphics[width=\linewidth]{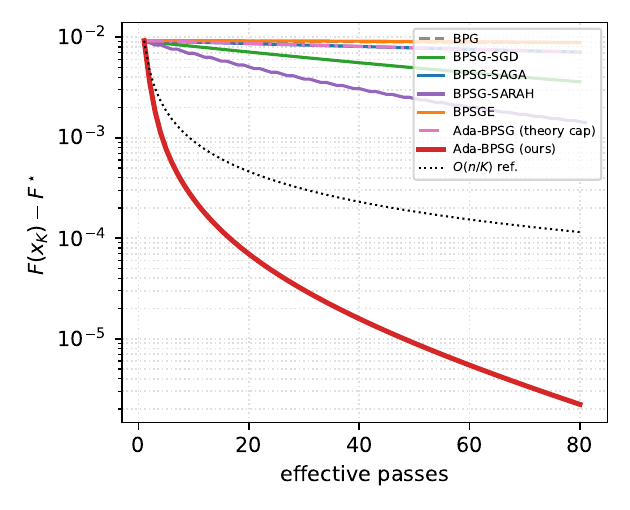}
        \caption{objective gap vs.\ effective passes}
    \end{subfigure}\hfill
    \begin{subfigure}[b]{0.31\textwidth}
        \centering        \includegraphics[width=\linewidth]{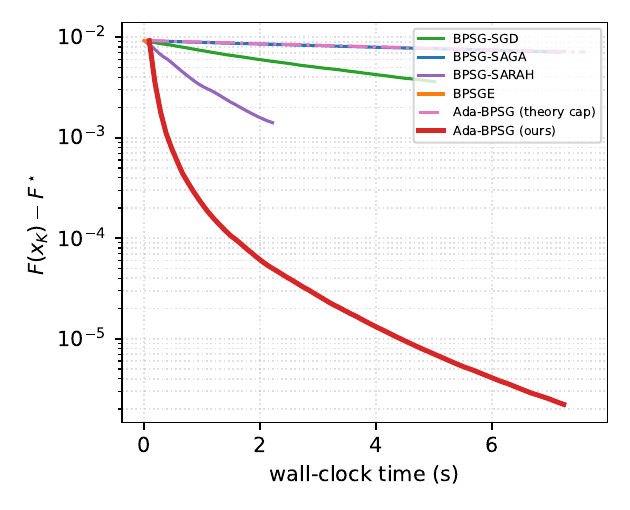}
        \caption{objective gap vs.\ wall-clock time}
    \end{subfigure}\hfill
    \begin{subfigure}[b]{0.31\textwidth}
        \centering
\includegraphics[width=\linewidth]{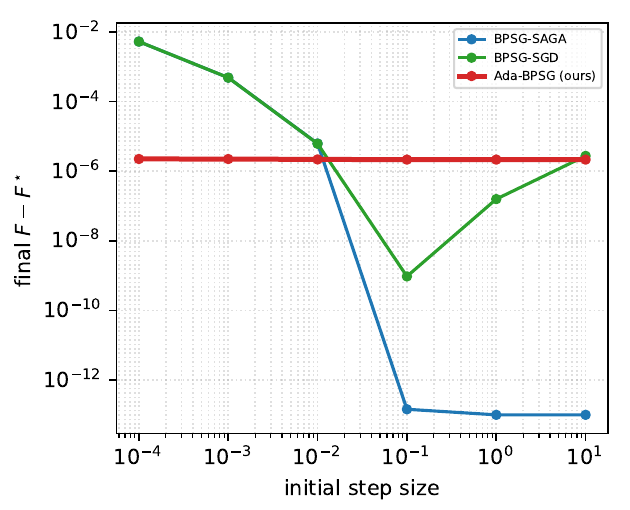}
        \caption{sensitivity to the initial step}
    \end{subfigure}
    \caption{Region-pooled Poisson unmixing on the Samson hyperspectral scene, read along the same axes as Figure~\ref{fig:simplex}. Fixed-step baselines use their a-priori guaranteed steps; ``Ada-BPSG (theory cap)'' is confined to the worst-case safeguard, ``Ada-BPSG (ours)'' relaxes it to the local component constant. (c) sweeps the \emph{initial} step---the step itself for the fixed methods, only the starting point for Ada-BPSG. Mean and standard deviation over five runs at $b=8$.}
    \label{fig:hsi}
\end{figure*}

Figure~\ref{fig:hsi} reproduces the synthetic behavior on real data. Held to the steps their own analysis follows the worst-case $M$, the fixed baselines stall between $10^{-2}$ and $10^{-3}$ (BPG $8.8\times10^{-3}$, SGD $3.6\times10^{-3}$, SAGA $7.1\times10^{-3}$, SARAH $1.4\times10^{-3}$), and the ``theory cap'' run of Ada-BPSG sits among them at $7.1\times10^{-3}$: the analysis-faithful step is once more conservative on the instance. Relaxing the safeguard to the local constant, Ada-BPSG drives the gap to $2.2\times10^{-6}$---more than two orders of magnitude below the best a-priori baseline---tracking the $O(n/K)$ reference slope and reaching $10^{-5}$ in a few seconds of wall-clock (Figure~\ref{fig:hsi} (b)). Figure~\ref{fig:hsi} (c) is transparent about the role of tuning: with its step hand-swept, fixed SAGA reaches near machine precision over a broad band of large steps, below Ada-BPSG, so on this instance a well-chosen constant step is very forgiving once found; the adaptive rule instead stays flat at $\approx2\times10^{-6}$ across five orders of magnitude of initial step, recovering a low gap with neither a per-instance search nor a line search, from the SAGA information it already maintains. The mechanism the synthetic experiment isolates thus governs a real unmixing problem under the identical configuration.

\subsection{Sparse nonnegative matrix factorization}
The third experiment tests the non-Euclidean update on sparse nonnegative matrix factorization (NMF). Given $\mathcal M\in\mathbb R^{I_1\times I_2}$ and a target rank $r<\min\{I_1,I_2\}$, we solve
\begin{equation}
    \begin{aligned}
        \min_{U,V}\quad
        &\Phi(U,V):=\frac12\|\mathcal M-UV\|_F^2,\\
        \text{subject to}\quad
        &U,V\geq0,\\
        &\|U_{:j}\|_0\leq s_1,\quad
        \|V_{j:}\|_0\leq s_2\quad(j\in[r]).
    \end{aligned}
\end{equation}
where $U\in\mathbb R_+^{I_1\times r}$, $V\in\mathbb R_+^{r\times I_2}$, $U_{:j}$ is column $j$ of $U$, $V_{j:}$ is row $j$ of $V$, and $\|\cdot\|_0$ counts nonzero entries. Equivalently, the constraints form the extended-valued term $h$, whose indicator is zero on the feasible set and $+\infty$ outside it.

We use the quartic kernel
\[
    \psi(U,V)=3\psi_1(U,V)+\|\mathcal M\|_F\psi_2(U,V),
\]
with
\[
    \psi_1(U,V):=\left(\frac{\|U\|_F^2+\|V\|_F^2}{2}\right)^2,
    \qquad
    \psi_2(U,V):=\frac{\|U\|_F^2+\|V\|_F^2}{2},
\]
which matches the relative-smooth geometry of matrix factorization \citep{mukkamala2019beyond}. Prior results permit the relative-smoothness scale $\bar L\geq1$ for this kernel \citep{bolte2018first,wang2024bregman}; we set $\bar L=1$. Because the cardinality constraints make $h$ nonconvex, this experiment evaluates the algorithmic mechanism beyond the convex-regularizer setting of Assumption \ref{assump}.

On the ORL dataset of 400 normalized face images, we set $r=25$, sample 5\% of the components per iteration, and test four sparsity levels. We compare Ada-BPSG with BPG and with BPSG/BPSGE using SGD, SAGA, or SARAH estimators. BPSGE is the extrapolated line-search variant. The smoothness scale $M$ is estimated by the power method used in \citet{wang2024bregman}, and the initial step size is selected from the same grid as in the classification experiments. Each method runs for 200 epochs.

\begin{figure*}[t!]
    \centering
    \begin{subfigure}[b]{0.235\textwidth}
        \centering
        \includegraphics[width=\linewidth]{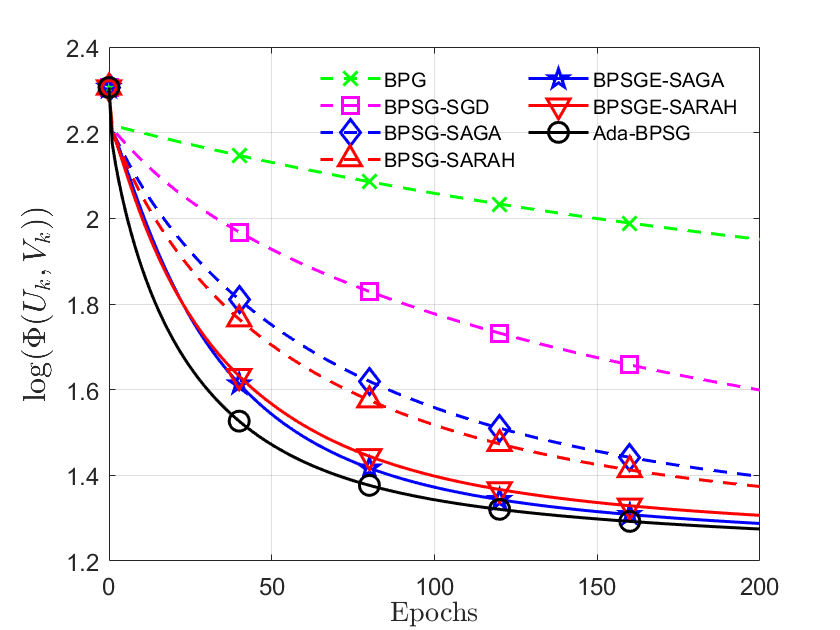}
        \caption{$s_1=I_1/2$ and $s_2=I_2/2$}
    \end{subfigure}\hfill
    \begin{subfigure}[b]{0.235\textwidth}
        \centering
        \includegraphics[width=\linewidth]{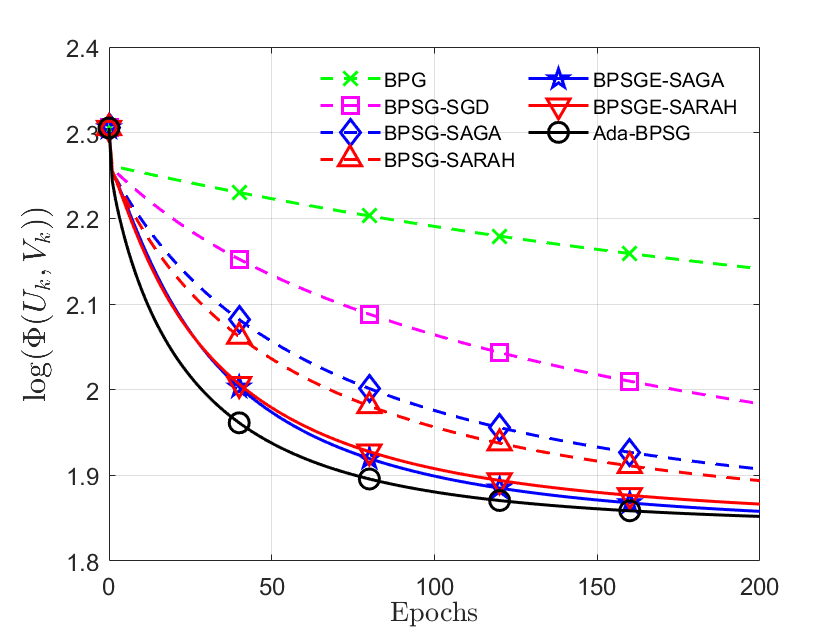}
        \caption{$s_1=I_1/4$ and $s_2=I_2/4$}
    \end{subfigure}\hfill
    \begin{subfigure}[b]{0.235\textwidth}
        \centering
        \includegraphics[width=\linewidth]{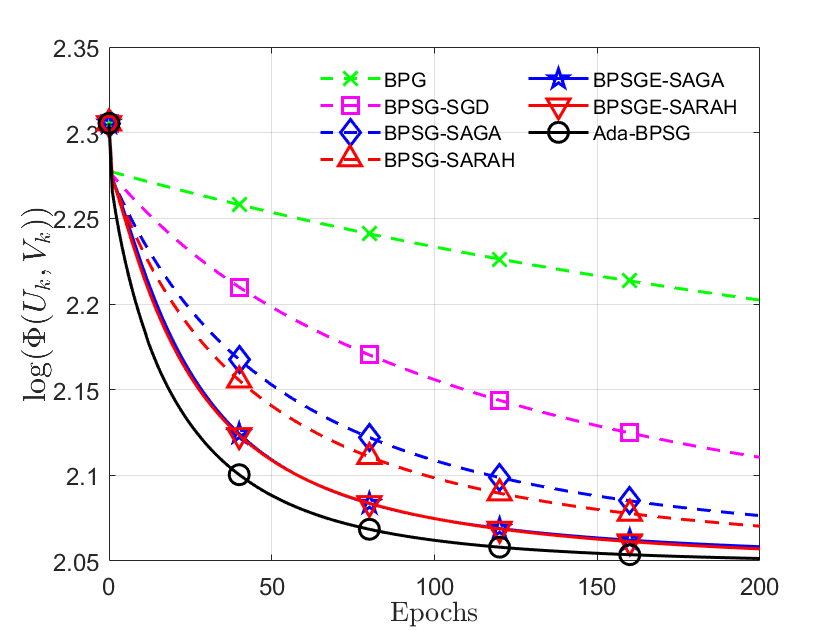}
        \caption{$s_1=I_1/6$ and $s_2=I_2/6$}
    \end{subfigure}\hfill
    \begin{subfigure}[b]{0.235\textwidth}
        \centering
        \includegraphics[width=\linewidth]{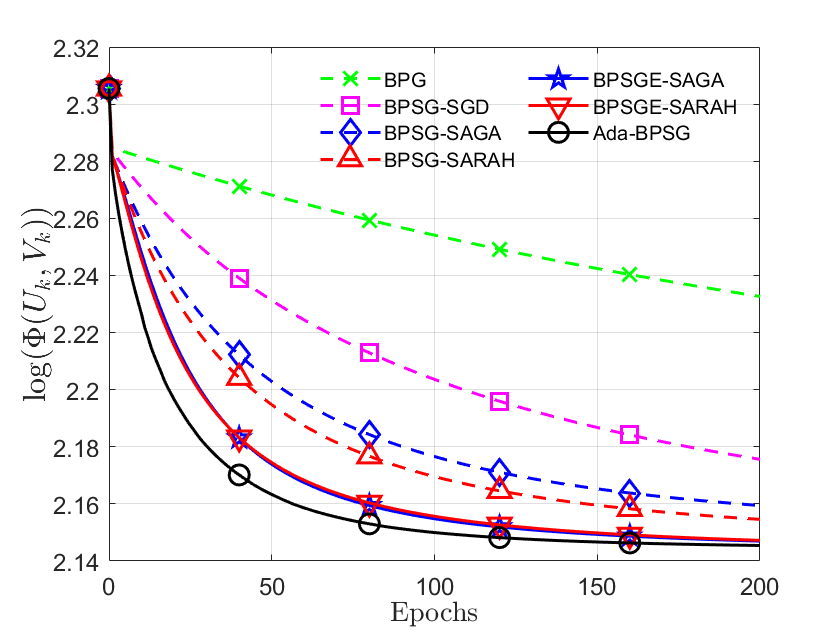}
        \caption{$s_1=I_1/8$ and $s_2=I_2/8$}
    \end{subfigure}
    \caption{Sparse NMF on ORL under four cardinality levels. Ada-BPSG attains the lowest objective trajectory in each regime, including against the extrapolated BPSGE methods that use line search.}
    \label{nmf}
\end{figure*}

Figure \ref{nmf} completes the empirical chain. Ada-BPSG descends faster than the stochastic BPSG variants and remains slightly below the two extrapolated BPSGE curves throughout the four sparsity regimes. Together with the Euclidean results, this shows that the stabilized adaptive step transfers from ordinary proximal SAGA to a problem-adapted Bregman geometry without introducing line-search evaluations.

\section{Conclusion}
Ada-BPSG turns the gradient information already maintained by SAGA into a stable, curvature-adaptive Bregman proximal step. Its mediant aggregation downweights unreliable local BB ratios, while clipping and monotonicity supply the deterministic invariant needed for convergence. This separation keeps the method line-search-free and makes the algorithmic and theoretical roles of the step-size components explicit.

The analysis follows a single chain across regimes: relative smoothness yields Bregman descent, local component smoothness controls the SAGA table error, and the safeguard balances the two. The resulting guarantees comprise an $O(n/K)$ convex rate, a restarted linear rate under relative quadratic growth, and an $O(1/K)$ bound on the expected squared proximal residual in the nonconvex setting. The experiments reflect the same mechanism. Ada-BPSG improves gradient efficiency and initial-step robustness in Euclidean logistic regression, exhibits the predicted convex rate on a simplex-constrained Poisson inverse problem under the entropy kernel---a non-Euclidean instance meeting every hypothesis of the analysis, where it adapts past the conservative worst-case safeguard to outrun every fixed-step baseline held to its guaranteed step, with the same advantage confirmed on real hyperspectral unmixing data---and achieves the lowest objective trajectories among the compared Bregman methods for sparse NMF, without the line search used by the extrapolated baselines.

\paragraph{Limitations and directions.}
The current theory analyzes the clipped sequence on a deterministic bounded region with a strongly convex kernel; the admissible safeguard depends on $\bar L$ and the local component constant $M$. General finite-sum models also retain the $O(nd)$ SAGA table, although linear-prediction losses admit the $O(n)$ implementation in Appendix A.2. Extending the same stabilized candidate to memory-light SVRG, SARAH, or SPIDER estimators, and deriving safeguards under weaker kernel and localization conditions, are natural next steps.

\clearpage
\bibliographystyle{plainnat}
\bibliography{reference}

\clearpage
\appendix
\section{Proofs and Supplementary Details}

\subsection{Auxiliary inequalities}
\begin{lemma}[Three-point proximal inequality] \label{lem:prox-three-point}
Let
\[
    x^+=\arg\min_x
    \left\{
        h(x)+\langle g,x\rangle+\frac1\eta D_\psi(x,x_k)
    \right\}.
\]
Then, for any $u\in\mathbf{E}$,
\[
    h(x^+)-h(u)+\langle g,x^+-u\rangle
    \leq
    \frac1\eta
    \left[
        D_\psi(u,x_k)-D_\psi(u,x^+)-D_\psi(x^+,x_k)
    \right].
\]
\end{lemma}
\begin{proof}
The optimality condition gives
\[
    0\in \partial h(x^+)+g+\frac1\eta(\nabla\psi(x^+)-\nabla\psi(x_k)).
\]
Choose $\xi^+\in\partial h(x^+)$ satisfying this inclusion. By convexity of $h$,
\[
    h(u)\geq h(x^+)+\langle \xi^+,u-x^+\rangle.
\]
Substituting $\xi^+=-g-\eta^{-1}(\nabla\psi(x^+)-\nabla\psi(x_k))$ and using the Bregman three-point identity gives the result.
\end{proof}

\subsection{Scalar-table implementation} \label{Efficient Implementation}
For a linear-prediction loss $f_i(x)=\Psi_i(\langle a_i,x\rangle)$, the scalar $\Phi_i^k:=\langle a_i,\phi_i^k\rangle$ determines the stored component gradient. Algorithm \ref{alg:impt} maintains these scalars together with the aggregate gradient $\mu_k$, reducing SAGA table storage from $O(nd)$ to $O(n)$ \citep{zhou2019direct}. The two secant accumulators remain scalar, so the adaptive step requires no additional vector table.

\begin{algorithm}[ht]
\caption{Efficient Implementation of Ada-BPSG}
\label{alg:impt}
\small
\begin{algorithmic}[1]

\State \textbf{Input:} max iterations $K$, batch-size $b\in\{1,\ldots,n\}$, initial point $x_0$
\Statex \hspace{\algorithmicindent} initial step-size $\eta_0\in[\eta_{\min},\eta_{\max}]$, update frequency $m\geq1$
\Statex \hspace{\algorithmicindent} BB scaling $\alpha>0$, BB buffer $\theta>0$, safeguard bounds $0<\eta_{\min}\leq\eta_{\max}$
\State \textbf{Initialize:} scalar table $\Phi_i^0 = \langle a_i,x_0 \rangle$ for $i=1,\ldots,n$
\Statex \hspace{\algorithmicindent} $\mu_0=\frac1n\sum_{i=1}^n\nabla\Psi_i(\Phi_i^0)a_i$, $\xi_0=0$, and $\delta_0=0$

\For{$k = 0,1,\dots,K-1$}

    \If{$k>0$}
        \State $\eta_k=\eta_{k-1}$
    \EndIf

    \If{$k>0$ and $k \bmod m = 0$}
        \State $\widehat{\eta}_k =
        \begin{cases}
        \dfrac{1}{\alpha}\dfrac{\xi_k}{\delta_k + \theta \xi_k}, & \xi_k>0,\\
        \eta_{k-1}, & \xi_k=0,
        \end{cases}$
        \State $\eta_k=\max\{\eta_{k-1},
        \operatorname{clip}_{[\eta_{\min},\eta_{\max}]}(\widehat{\eta}_k)\}$
        \State $\xi_k = 0$, $\delta_k = 0$
    \EndIf

    \State Randomly pick a batch $B_k$ uniformly without replacement.
    \State Store $\Phi^{k+1}_i = \langle a_i, x_k \rangle$ for $i \in B_k$
    \State Set $\Phi_i^{k+1}=\Phi_i^k$ for $i\notin B_k$

    \State $\widetilde{\nabla}_k =
        \frac{1}{b}\sum_{i\in B_k}
        \bigl(
            \nabla \Psi_i(\Phi^{k+1}_i)
            -
            \nabla \Psi_i(\Phi^k_i)
        \bigr)\, a_i
        + \mu_k$

    \State $x_{k+1}
        = \arg\min_{x}\left\{
            h(x)
            + \langle \widetilde{\nabla}_{k}, x \rangle
            + \frac{1}{\eta_k} D_{\psi}(x,x_k)
        \right\}$

    \State $\xi_{k+1} = \xi_k +
        \sum_{i \in B_k}
        \left|
            (\Phi_i^{k+1}-\Phi_i^k)
            \bigl(
                \nabla \Psi_i(\Phi^{k+1}_i)
                -
                \nabla \Psi_i(\Phi^{k}_i)
            \bigr)
        \right|$

    \State $\delta_{k+1} = \delta_k +
        \sum_{i \in B_k}
        \Bigl(
            \bigl(
                \nabla \Psi_i(\Phi^{k+1}_i)
                -
                \nabla \Psi_i(\Phi^{k}_i)
            \bigr)^2
            \cdot \|a_i\|_*^2
        \Bigr)$

    \State $\mu_{k+1}
        = \mu_k
        + \frac{1}{n}\sum_{i\in B_k}
        \bigl(
            \nabla \Psi_i(\Phi^{k+1}_i)
            -
            \nabla \Psi_i(\Phi^k_i)
        \bigr)\, a_i$

\EndFor

\State \textbf{return} $x_K$

\end{algorithmic}
\end{algorithm}

\subsection{Proof of Lemma \ref{VB}}
Let
\[
    a_i^k:=\nabla f_i(x_k)-\nabla f_i(\phi_i^k),
    \qquad
    \bar a^k:=\frac1n\sum_{i=1}^n a_i^k.
\]
Since $\mu_k=\frac1n\sum_i\nabla f_i(\phi_i^k)$, the estimator can be written as
\[
    \widetilde{\nabla}_k
    =
    \frac1b\sum_{i\in B_k}a_i^k+\mu_k.
\]
The uniform sampling assumption gives
\[
    \mathbb{E}_k[\widetilde{\nabla}_k]
    =
    \bar a^k+\mu_k
    =
    \nabla f(x_k).
\]
Moreover,
\[
    \widetilde{\nabla}_k-\nabla f(x_k)
    =
    \frac1b\sum_{i\in B_k}(a_i^k-\bar a^k).
\]
By the sampling second-moment bound in Assumption \ref{assump},
\[
\begin{split}
    \mathbb{E}_k\|\widetilde{\nabla}_k-\nabla f(x_k)\|_*^2
    &\leq
    \frac{\sigma_b}{bn}\sum_{i=1}^n\|a_i^k-\bar a^k\|_*^2\\
    &\leq
    \frac{4\sigma_b}{bn}\sum_{i=1}^n\|a_i^k\|_*^2.
\end{split}
\]
For the convex setting, the self-bounding condition implies
\[
\begin{split}
    \|a_i^k\|_*^2
    &\leq
    2\|\nabla f_i(x_k)-\nabla f_i(x^*)\|_*^2
    +2\|\nabla f_i(\phi_i^k)-\nabla f_i(x^*)\|_*^2\\
    &\leq
    4M\left(\Delta_i(x_k)+\Delta_i(\phi_i^k)\right).
\end{split}
\]
Combining the two displays yields
\[
    \mathbb{E}_k\|\widetilde{\nabla}_k-\nabla f(x_k)\|_*^2
    \leq
    \frac{16\sigma_bM}{b}
    \left[
        \Delta(x_k)+\frac1n\sum_{i=1}^n\Delta_i(\phi_i^k)
    \right].
\]
\qed

\subsection{Proof of Theorem \ref{theo1}}
Let
\[
    Q_i(x):=\Delta_i(x)+H(x),
    \qquad
    T_k:=\frac1\omega\frac1n\sum_{i=1}^n Q_i(\phi_i^k),
    \qquad
    \omega:=\frac{4b}{n}.
\]
Since $\phi_i^0=x_0$ and $\nabla f(x^*)+\rho^*=0$,
\[
    T_0=\frac{n}{4b}\left(F(x_0)-F(x^*)\right).
\]
The table update gives
\[
    \mathbb{E}_k[T_{k+1}]
    =
    \left(1-\frac bn\right)T_k
    +\frac{b}{n\omega}\left(F(x_k)-F(x^*)\right).
\]

Using the upper smooth-adaptability of $f$ and Lemma \ref{lem:prox-three-point} with $u=x^*$ and $g=\widetilde{\nabla}_k$, we first obtain
\[
\begin{split}
    F(x_{k+1})-F(x^*)
    \leq{}&
    f(x_k)-f(x^*)-\langle\nabla f(x_k),x_k-x^*\rangle\\
    &+\langle \nabla f(x_k)-\widetilde{\nabla}_k,x_{k+1}-x_k\rangle
    +\langle \nabla f(x_k)-\widetilde{\nabla}_k,x_k-x^*\rangle\\
    &+\left(\bar L-\frac1{\eta_k}\right)D_\psi(x_{k+1},x_k)
    +\frac1{\eta_k}\left(D_\psi(x^*,x_k)-D_\psi(x^*,x_{k+1})\right).
\end{split}
\]
By convexity of $f$, the first line is nonpositive. Taking conditional expectation and using
$\mathbb{E}_k[\widetilde{\nabla}_k]=\nabla f(x_k)$ gives
\[
\begin{split}
    \mathbb{E}_k[F(x_{k+1})-F(x^*)]
    \leq{}&
    \mathbb{E}_k\langle \nabla f(x_k)-\widetilde{\nabla}_k,x_{k+1}-x_k\rangle\\
    &+\left(\bar L-\frac1{\eta_k}\right)\mathbb{E}_kD_\psi(x_{k+1},x_k)
    +\frac1{\eta_k}\left(D_\psi(x^*,x_k)-\mathbb{E}_kD_\psi(x^*,x_{k+1})\right).
\end{split}
\]
By Young's inequality and $D_\psi(x_{k+1},x_k)\geq \frac12\|x_{k+1}-x_k\|^2$,
\[
    \mathbb{E}_k\langle \nabla f(x_k)-\widetilde{\nabla}_k,x_{k+1}-x_k\rangle
    \leq
    \frac{1}{2\beta}
    \mathbb{E}_k\|\widetilde{\nabla}_k-\nabla f(x_k)\|_*^2
    +\beta\mathbb{E}_kD_\psi(x_{k+1},x_k).
\]
Take $\beta=32\sigma_bM/b$. Since $\eta_{\max}\leq(\bar L+\beta)^{-1}$, the coefficient of $D_\psi(x_{k+1},x_k)$ is nonpositive. Lemma \ref{VB} gives
\[
    \frac{1}{2\beta}
    \mathbb{E}_k\|\widetilde{\nabla}_k-\nabla f(x_k)\|_*^2
    \leq
    \frac14\Delta(x_k)
    +\frac14\frac1n\sum_{i=1}^n\Delta_i(\phi_i^k).
\]
Because $H\geq0$, we have $\Delta(x_k)\leq F(x_k)-F(x^*)$ and
$\frac1n\sum_i\Delta_i(\phi_i^k)\leq\omega T_k$. Combining these bounds with the recursion for $T_{k+1}$, and using $b/(n\omega)=1/4$ and $\omega/4=b/n$, gives
\[
\begin{split}
    \mathbb{E}_k[F(x_{k+1})-F(x^*)+T_{k+1}]
    \leq{}&
    T_k+\frac12(F(x_k)-F(x^*))\\
    &+\frac1{\eta_k}
    \left(D_\psi(x^*,x_k)-\mathbb{E}_kD_\psi(x^*,x_{k+1})\right).
\end{split}
\]
Rearranging,
\[
\begin{split}
    \frac12\mathbb{E}_k[F(x_{k+1})-F(x^*)]
    \leq{}&
    T_k-\mathbb{E}_kT_{k+1}
    +\frac12\left(F(x_k)-F(x^*)-\mathbb{E}_k[F(x_{k+1})-F(x^*)]\right)\\
    &+\frac1{\eta_k}
    \left(D_\psi(x^*,x_k)-\mathbb{E}_kD_\psi(x^*,x_{k+1})\right).
\end{split}
\]
Let $a_k=1/\eta_k$. Since the convex theorem assumes $\eta_k$ is nondecreasing, $a_k$ is nonincreasing. Therefore,
\[
    \sum_{k=0}^{K-1}
    \mathbb{E}\left[
    a_k\left(D_\psi(x^*,x_k)-D_\psi(x^*,x_{k+1})\right)
    \right]
    \leq
    a_0D_\psi(x^*,x_0)
    \leq
    \frac1{\eta_{\min}}D_\psi(x^*,x_0).
\]
Summing the preceding one-step inequality and using $T_K\geq0$ yields
\[
    \frac12\sum_{k=0}^{K-1}
    \mathbb{E}[F(x_{k+1})-F(x^*)]
    \leq
    T_0+\frac12(F(x_0)-F(x^*))
    +\frac1{\eta_{\min}}D_\psi(x^*,x_0).
\]
Jensen's inequality for $\bar x_K=\frac1K\sum_{k=1}^K x_k$ completes the proof.

\subsection{Proof of Theorem \ref{theo2}}
Condition on the history up to the beginning of restart round $t$. The restart round reinitializes the SAGA table by setting all table points to $x^{(t)}$ and uses a safeguarded step-size sequence satisfying the assumptions of Theorem \ref{theo1} within that round. Therefore the initial Lyapunov term in Theorem \ref{theo1} has the same form with $x_0$ replaced by $x^{(t)}$. Applying Theorem \ref{theo1} to that round and using the relative quadratic growth condition gives
\[
    D_\psi(x^*,x^{(t)})
    \leq
    \frac1\mu\left(F(x^{(t)})-F(x^*)\right).
\]
Hence, with $A_b=\frac{n}{4b}+\frac12$,
\[
    \mathbb{E}\left[F(x^{(t+1)})-F(x^*)\mid x^{(t)}\right]
    \leq
    \frac{2}{K}
    \left(A_b+\frac1{\mu\eta_{\min}}\right)
    \left(F(x^{(t)})-F(x^*)\right).
\]
Taking total expectation, the stated lower bound on $K$ makes the contraction factor at most $1/2$, and induction gives the claim.

\subsection{Proof of Lemma \ref{nonconvexVariance}}
Using the notation from the proof of Lemma \ref{VB},
\[
    \widetilde{\nabla}_k-\nabla f(x_k)
    =
    \frac1b\sum_{i\in B_k}(a_i^k-\bar a^k).
\]
The sampling second-moment bound and Jensen's inequality give
\[
    \mathbb{E}_k\|\widetilde{\nabla}_k-\nabla f(x_k)\|_*^2
    \leq
    \frac{4\sigma_b}{bn}\sum_{i=1}^n\|a_i^k\|_*^2.
\]
By component Lipschitz continuity on the deterministic set $D$,
\[
    \|a_i^k\|_*^2
    =
    \|\nabla f_i(x_k)-\nabla f_i(\phi_i^k)\|_*^2
    \leq
    M^2\|x_k-\phi_i^k\|^2.
\]
This proves the lemma.

\subsection{Proof of Theorem \ref{theo3}}
Let
\[
    \widetilde{x}_{k+1}
    =
    \operatorname{Prox}^{\psi}_{h,\eta_k}(x_k,\nabla f(x_k)).
\]
Let $e_k:=\nabla f(x_k)-\widetilde{\nabla}_k$. Using upper smooth-adaptability of $f$ and Lemma \ref{lem:prox-three-point} for the stochastic prox step with $u=\widetilde{x}_{k+1}$ gives
\[
\begin{split}
    F(x_{k+1})
    \leq{}&
    f(x_k)+h(\widetilde{x}_{k+1})
    +\langle\nabla f(x_k),x_{k+1}-x_k\rangle
    -\langle\widetilde{\nabla}_k,x_{k+1}-\widetilde{x}_{k+1}\rangle\\
    &+\left(\bar L-\frac1{\eta_k}\right)D_\psi(x_{k+1},x_k)
    +\frac1{\eta_k}D_\psi(\widetilde{x}_{k+1},x_k)
    -\frac1{\eta_k}D_\psi(\widetilde{x}_{k+1},x_{k+1}).
\end{split}
\]
Applying Lemma \ref{lem:prox-three-point} to the full-gradient prox step with $u=x_k$ gives
\[
    h(\widetilde{x}_{k+1})-h(x_k)
    +\langle\nabla f(x_k),\widetilde{x}_{k+1}-x_k\rangle
    \leq
    -\frac1{\eta_k}D_\psi(x_k,\widetilde{x}_{k+1})
    -\frac1{\eta_k}D_\psi(\widetilde{x}_{k+1},x_k).
\]
Combining the last two displays yields
\[
\begin{split}
    F(x_{k+1})
    \leq{}&
    F(x_k)+\langle e_k,x_{k+1}-\widetilde{x}_{k+1}\rangle
    +\left(\bar L-\frac1{\eta_k}\right)D_\psi(x_{k+1},x_k)\\
    &-\frac1{\eta_k}D_\psi(x_k,\widetilde{x}_{k+1})
    -\frac1{\eta_k}D_\psi(\widetilde{x}_{k+1},x_{k+1}).
\end{split}
\]
Young's inequality and the strong convexity of $\psi$ imply
\[
    \langle e_k,x_{k+1}-\widetilde{x}_{k+1}\rangle
    \leq
    \frac{\eta_k}{2}\|e_k\|_*^2
    +\frac1{\eta_k}D_\psi(\widetilde{x}_{k+1},x_{k+1}).
\]
Taking conditional expectation and applying Lemma \ref{nonconvexVariance}, we obtain
\[
\begin{split}
    \mathbb{E}_k[F(x_{k+1})]
    \leq{}&
    F(x_k)
    -\frac1{\eta_k}D_\psi(x_k,\widetilde{x}_{k+1})
    +\left(\bar L-\frac1{\eta_k}\right)
    \mathbb{E}_kD_\psi(x_{k+1},x_k)\\
    &+\frac{2\sigma_b\eta_kM^2}{bn}
    \sum_{i=1}^n\|x_k-\phi_i^k\|^2.
\end{split}
\]
Define
\[
    S_k:=\frac1n\sum_{i=1}^n\|x_k-\phi_i^k\|^2,
    \qquad
    p:=\frac{b}{n},
\]
and use the deterministic Lyapunov coefficient
\[
    c:=\frac{4\sigma_b\eta_{\max}M^2n}{b^2}.
\]
Let $d_k:=x_{k+1}-x_k$ and $r_i^k:=x_k-\phi_i^k$. For a fixed sampled batch $B_k$, the table update gives
\[
    x_{k+1}-\phi_i^{k+1}
    =
    \begin{cases}
        d_k, & i\in B_k,\\
        d_k+r_i^k, & i\notin B_k.
    \end{cases}
\]
For any $\gamma>0$, Young's inequality yields
\[
    \|d_k+r_i^k\|^2
    \leq
    \left(1+\frac1\gamma\right)\|d_k\|^2
    +(1+\gamma)\|r_i^k\|^2.
\]
Consequently,
\[
\begin{split}
    S_{k+1}
    &=
    \frac1n\sum_{i\in B_k}\|d_k\|^2
    +\frac1n\sum_{i\notin B_k}\|d_k+r_i^k\|^2\\
    &\leq
    \left(1+\frac{1-p}{\gamma}\right)\|d_k\|^2
    +\frac{1+\gamma}{n}\sum_{i\notin B_k}\|r_i^k\|^2.
\end{split}
\]
Taking conditional expectation and using $D_\psi(x_{k+1},x_k)\geq\frac12\|d_k\|^2$ and
$\mathbb{E}_k[\mathbf{1}_{\{i\notin B_k\}}]=1-p$, we obtain
\begin{equation} \label{eq:table-distance-recursion}
    \mathbb{E}_k[S_{k+1}]
    \leq
    2\left(1+\frac{1-p}{\gamma}\right)
    \mathbb{E}_kD_\psi(x_{k+1},x_k)
    +(1+\gamma)(1-p)S_k.
\end{equation}
Combining (\ref{eq:table-distance-recursion}) with the one-step bound for $F(x_{k+1})$ gives
\[
\begin{split}
    \mathbb{E}_k[F(x_{k+1})+cS_{k+1}]
    \leq{}&
    F(x_k)
    -\frac1{\eta_k}D_\psi(x_k,\widetilde{x}_{k+1})\\
    &+
    \left[
        \bar L-\frac1{\eta_k}
        +2c\left(1+\frac{1-p}{\gamma}\right)
    \right]\mathbb{E}_kD_\psi(x_{k+1},x_k)\\
    &+
    \left[
        c(1+\gamma)(1-p)
        +\frac{2\sigma_b\eta_kM^2}{b}
    \right]S_k.
\end{split}
\]
Choose $\gamma=p/2$. Since $(1+\gamma)(1-p)\leq1-p/2$ and $\eta_k\leq\eta_{\max}$, the definition of $c$ gives
\[
    c(1+\gamma)(1-p)
    +\frac{2\sigma_b\eta_kM^2}{b}
    \leq c.
\]
Moreover,
\[
    2\left(1+\frac{1-p}{\gamma}\right)
    =
    \frac{2(2-p)}{p}
    =
    \frac{2(2n-b)}{b}.
\]
Therefore,
\[
    2c\left(1+\frac{1-p}{\gamma}\right)
    \leq
    \frac{8\sigma_b\eta_{\max}M^2n(2n-b)}{b^3}
    =
    \Gamma_b\eta_{\max}.
\]
The condition in Theorem \ref{theo3} implies
\[
    \bar L-\frac1{\eta_k}
    +2c\left(1+\frac{1-p}{\gamma}\right)
    \leq0
\]
for every $k$, because $\eta_k\leq\eta_{\max}$. Hence
\[
    \mathbb{E}_k[F(x_{k+1})+cS_{k+1}]
    \leq
    F(x_k)+cS_k
    -\frac1{\eta_k}D_\psi(x_k,\widetilde{x}_{k+1}).
\]
Taking total expectation and summing gives, using $S_0=0$, $S_K\geq0$, and $F(x_K)\geq F_{\inf}$,
\[
    \sum_{k=0}^{K-1}
    \mathbb{E}\left[
        \frac1{\eta_k}D_\psi(x_k,\widetilde{x}_{k+1})
    \right]
    \leq
    F(x_0)-F_{\inf}.
\]
Since $D_\psi(u,v)\geq\frac12\|u-v\|^2$ and $\eta_k\leq\eta_{\max}$,
\[
    \frac{1}{2\eta_{\max}}
    \sum_{k=0}^{K-1}\mathbb{E}\|x_k-\widetilde{x}_{k+1}\|^2
    \leq
    F(x_0)-F_{\inf}.
\]
For $R$ sampled independently and uniformly from $\{0,\ldots,K-1\}$ and $\mathcal R_R=(x_R-\widetilde{x}_{R+1})/\eta_{\max}$, this yields
\[
    \mathbb{E}\|\mathcal R_R\|^2
    \leq
    \frac{2(F(x_0)-F_{\inf})}{K\eta_{\max}}.
\]

\section{Additional experiments} \label{Additional experiments}
\subsection{Additional experiments for logistic loss}
Figures \ref{fig: logistic loss bz1}--\ref{fig: logistic loss bz64} extend the initial-step experiment in Figure \ref{fig:grad_norm-step_size} to batch sizes $1$, $16$, and $64$. Ada-BPSG is denoted by SAGA-BB in this Euclidean setting.
\begin{figure}[tp!]
    \centering
    \begin{subfigure}[b]{0.235\textwidth}
        \centering
        \includegraphics[width=\linewidth]{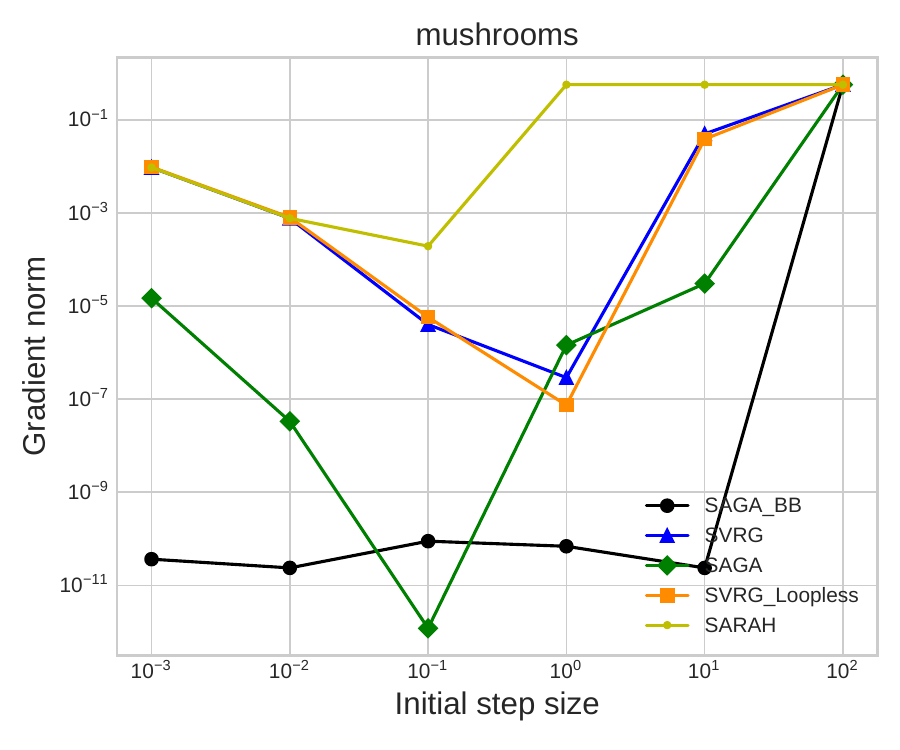}
        \caption{}
    \end{subfigure}\hfill
    \begin{subfigure}[b]{0.235\textwidth}
        \centering
        \includegraphics[width=\linewidth]{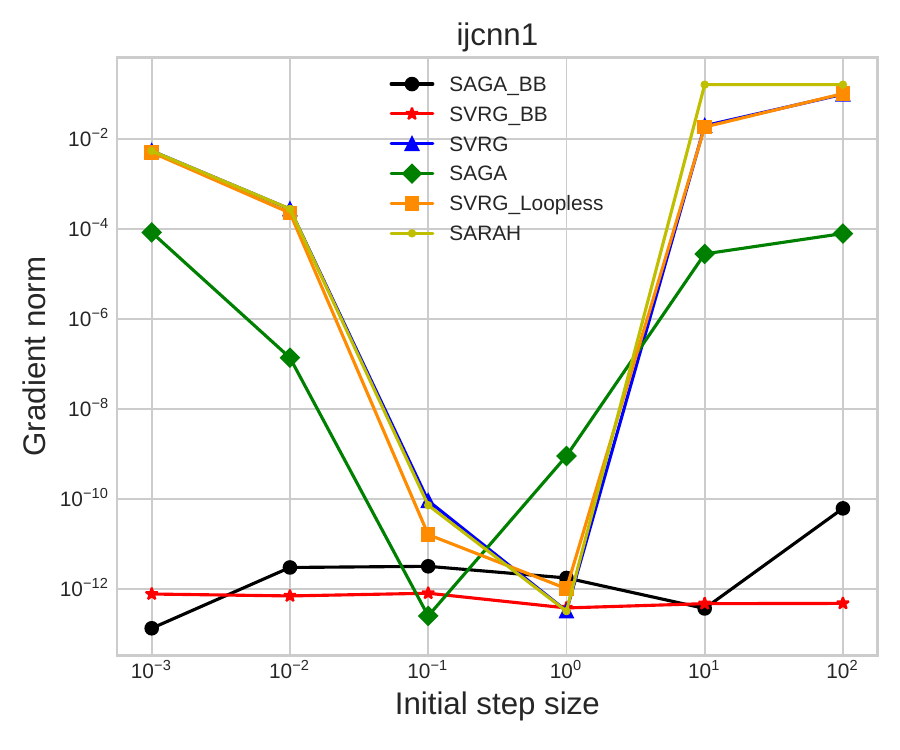}
        \caption{}
    \end{subfigure}\hfill
    \begin{subfigure}[b]{0.235\textwidth}
        \centering
        \includegraphics[width=\linewidth]{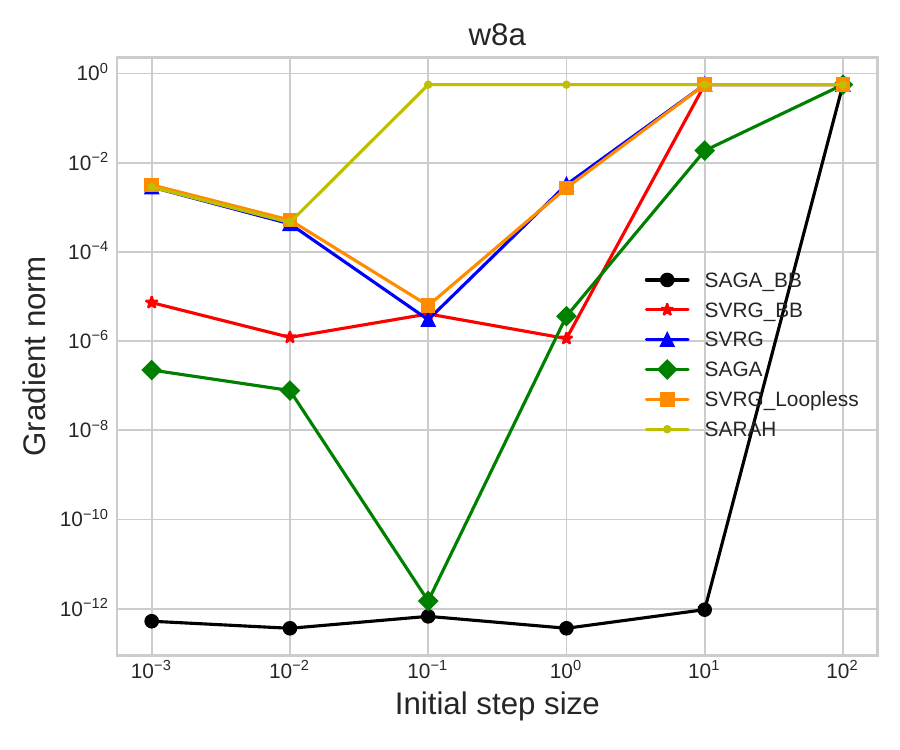}
        \caption{}
    \end{subfigure}\hfill
    \begin{subfigure}[b]{0.235\textwidth}
        \centering
        \includegraphics[width=\linewidth]{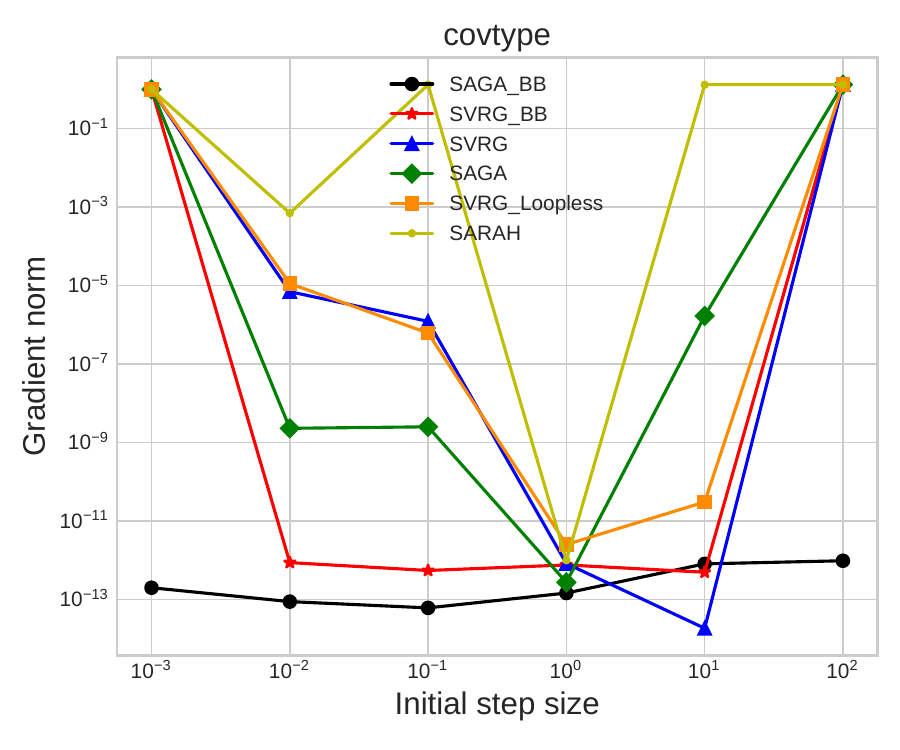}
        \caption{}
    \end{subfigure}
    \caption{Initial-step sensitivity for logistic loss with batch size $b=1$. Gradient norms are capped at $1$ for readability.}
    \label{fig: logistic loss bz1}
\end{figure}

\begin{figure}[tp!]
    \centering
    \begin{subfigure}[b]{0.235\textwidth}
        \centering
        \includegraphics[width=\linewidth]{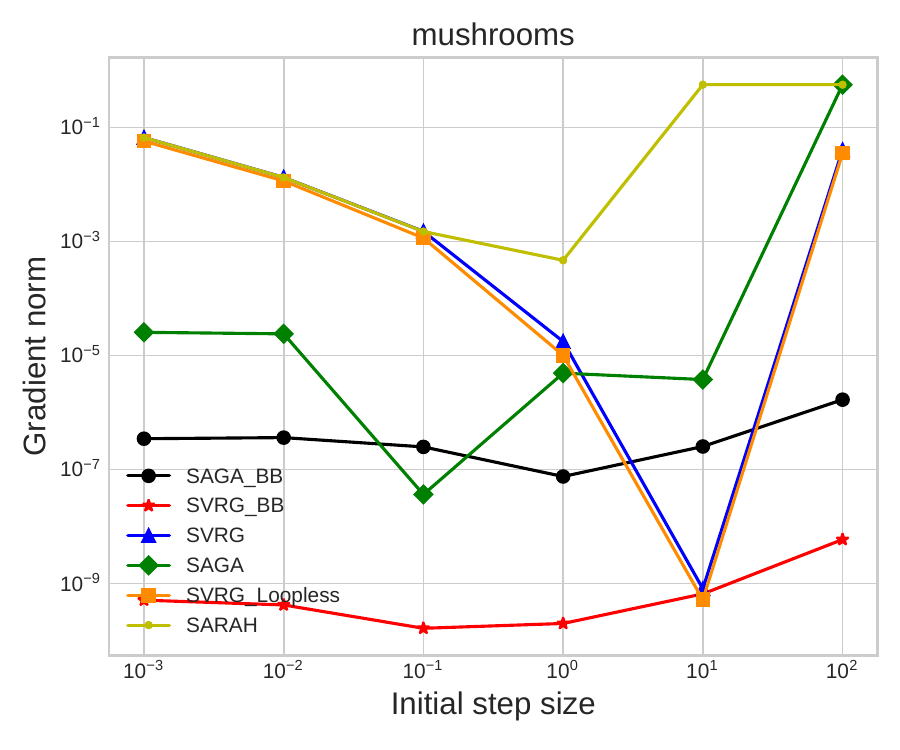}
        \caption{}
    \end{subfigure}\hfill
    \begin{subfigure}[b]{0.235\textwidth}
        \centering
        \includegraphics[width=\linewidth]{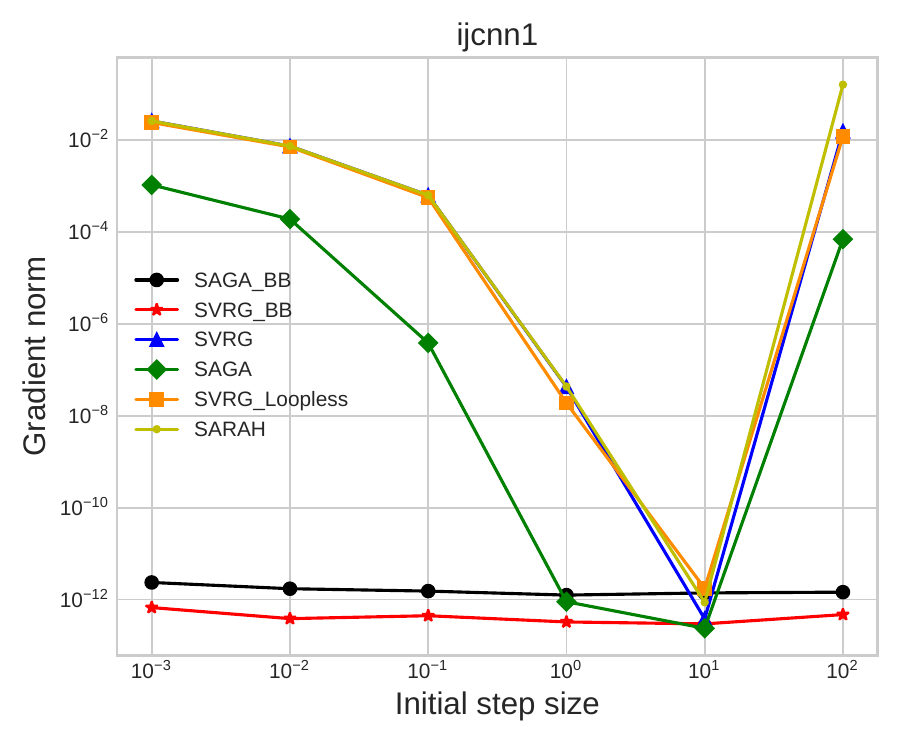}
        \caption{}
    \end{subfigure}\hfill
    \begin{subfigure}[b]{0.235\textwidth}
        \centering
        \includegraphics[width=\linewidth]{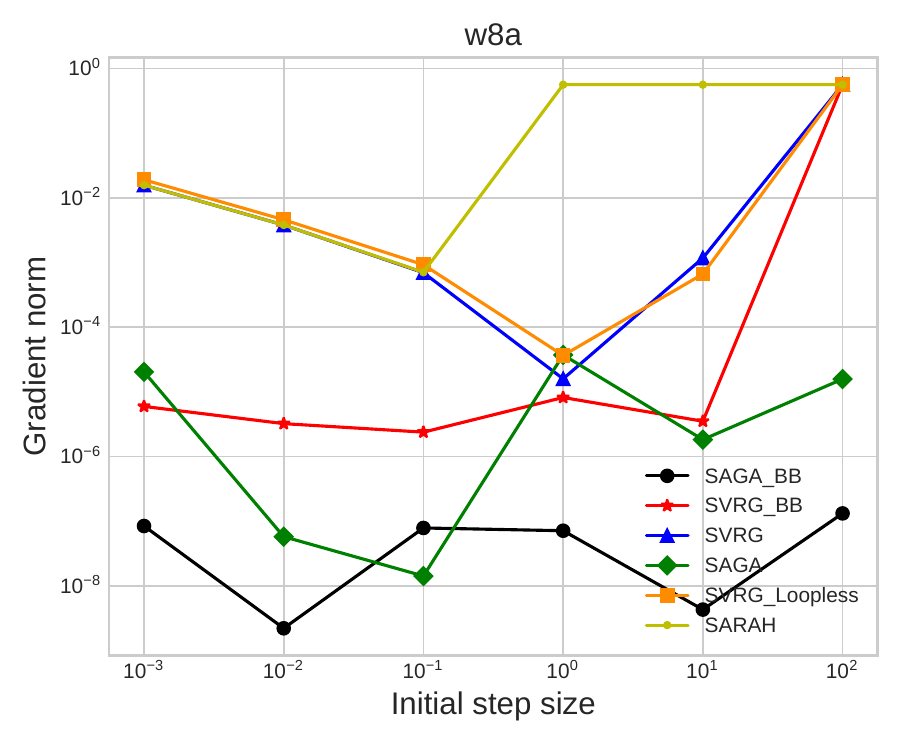}
        \caption{}
    \end{subfigure}\hfill
    \begin{subfigure}[b]{0.235\textwidth}
        \centering
        \includegraphics[width=\linewidth]{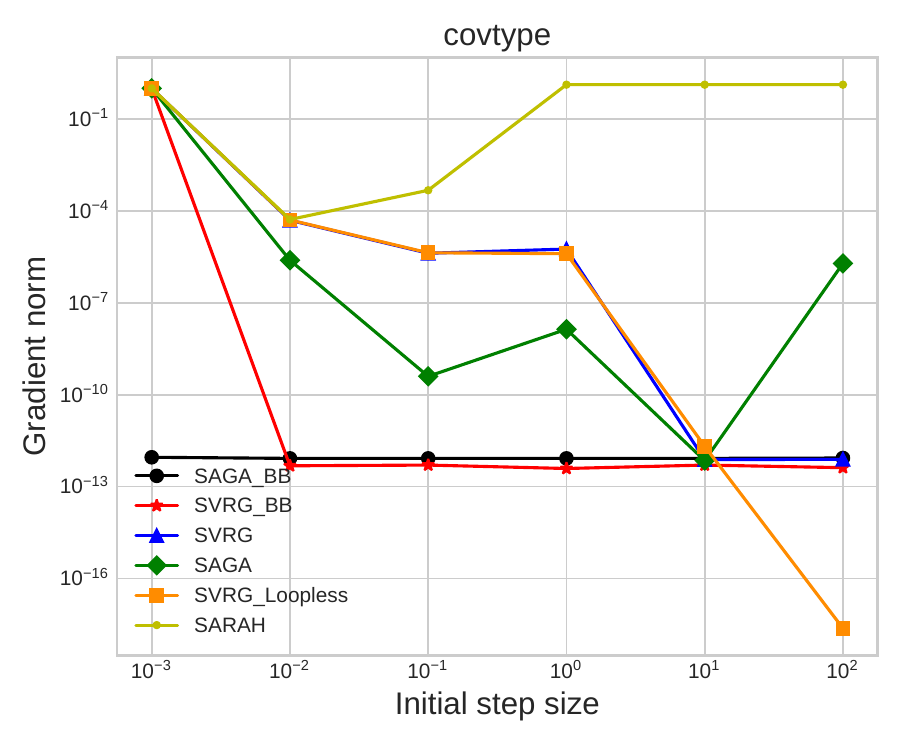}
        \caption{}
    \end{subfigure}
    \caption{Initial-step sensitivity for logistic loss with batch size $b=16$. Gradient norms are capped at $1$ for readability.}
    \label{fig: logistic loss bz16}
\end{figure}

\begin{figure}[tp!]
    \centering
    \begin{subfigure}[b]{0.235\textwidth}
        \centering
        \includegraphics[width=\linewidth]{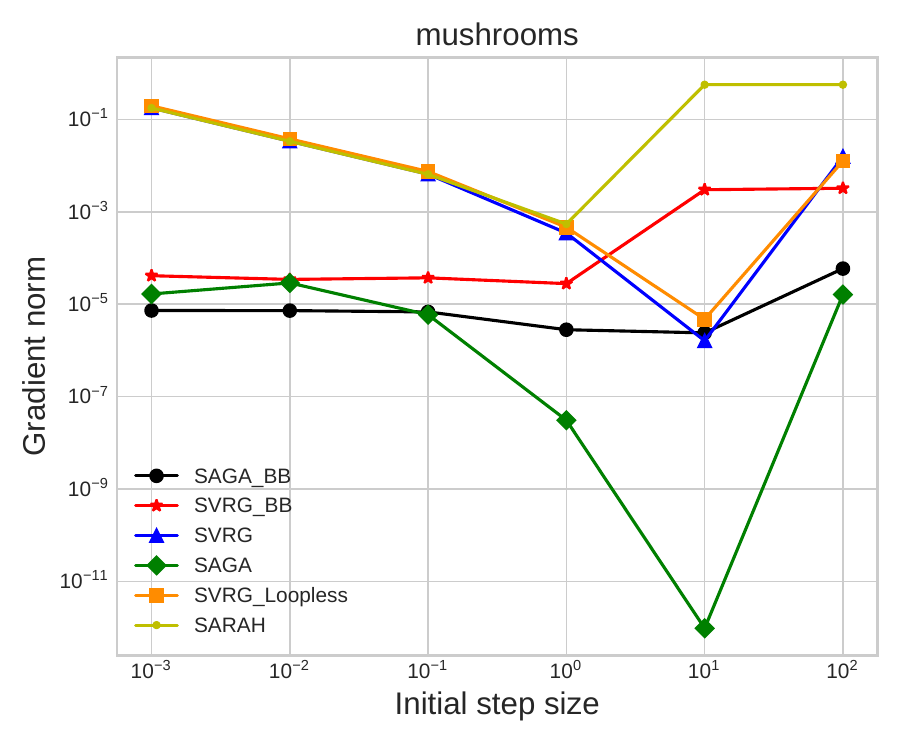}
        \caption{}
    \end{subfigure}\hfill
    \begin{subfigure}[b]{0.235\textwidth}
        \centering
        \includegraphics[width=\linewidth]{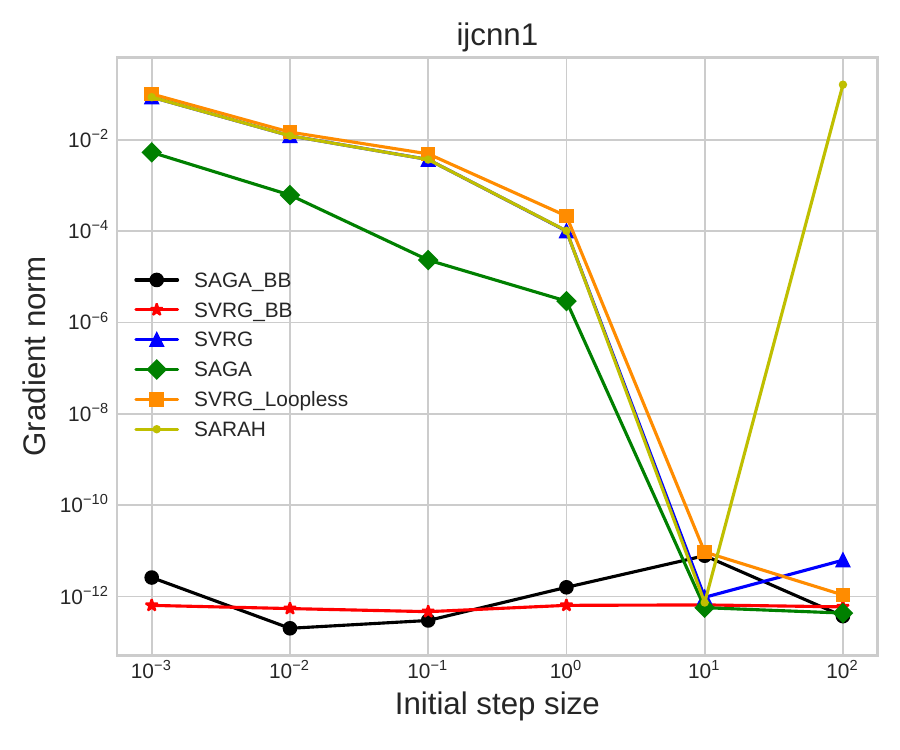}
        \caption{}
    \end{subfigure}\hfill
    \begin{subfigure}[b]{0.235\textwidth}
        \centering
        \includegraphics[width=\linewidth]{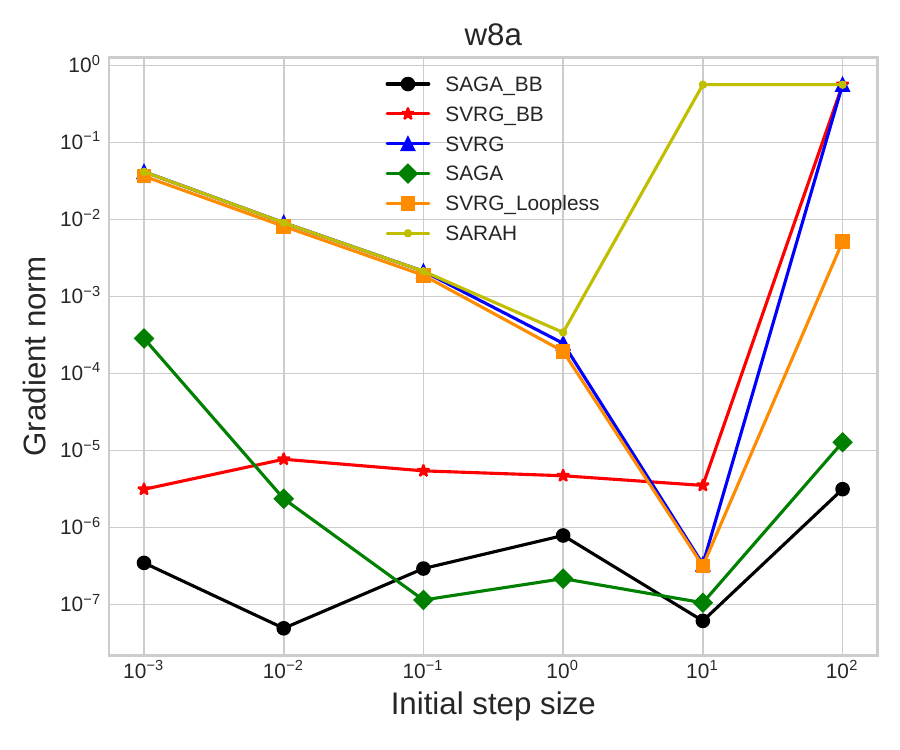}
        \caption{}
    \end{subfigure}\hfill
    \begin{subfigure}[b]{0.235\textwidth}
        \centering
        \includegraphics[width=\linewidth]{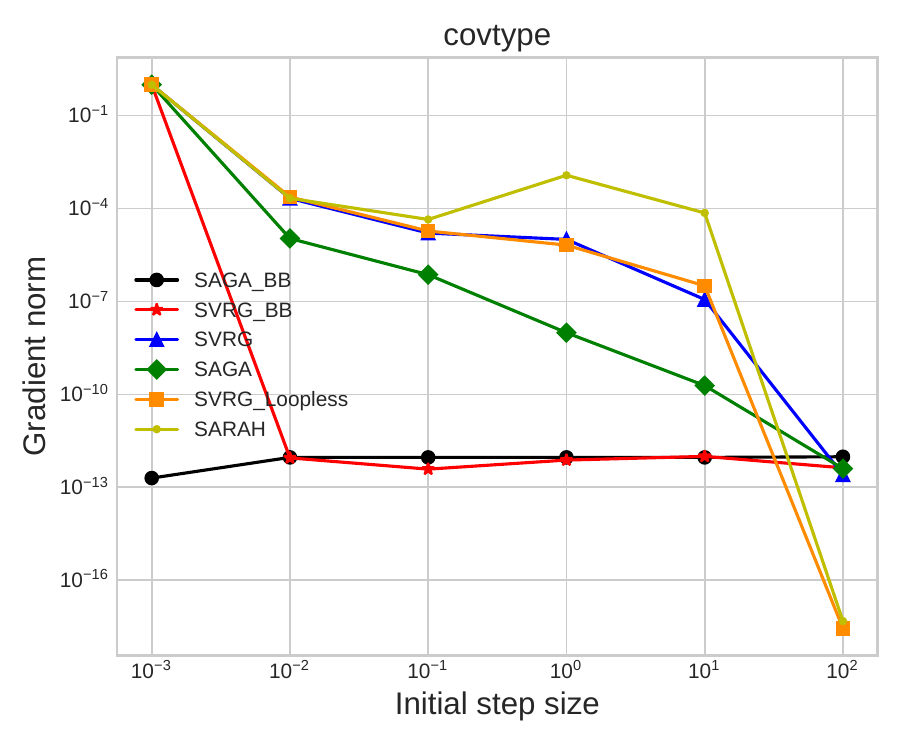}
        \caption{}
    \end{subfigure}
    \caption{Initial-step sensitivity for logistic loss with batch size $b=64$. Gradient norms are capped at $1$ for readability.}
    \label{fig: logistic loss bz64}
\end{figure}

Across these batch sizes, SAGA-BB preserves the broad low-gradient region observed at $b=8$, whereas the nonadaptive baselines remain sensitive to the initial scale. This consistency supports the intended role of the adaptive refresh: it reduces dependence on both the starting step and the number of component updates aggregated at each iteration.

\subsection{Additional experiments for Huber loss}\label{sub: Huber loss}
We next replace logistic loss by Huber loss in the same $\ell_2$-regularized binary-classification problems. The Huber penalty is quadratic for residuals near zero and linear beyond a fixed threshold, combining local smoothness with reduced sensitivity to large residuals. This change tests whether the step-size behavior persists across loss geometries as well as batch sizes.

\begin{figure}[tp!]
    \centering
    \begin{subfigure}[b]{0.235\textwidth}
        \centering
        \includegraphics[width=\linewidth]{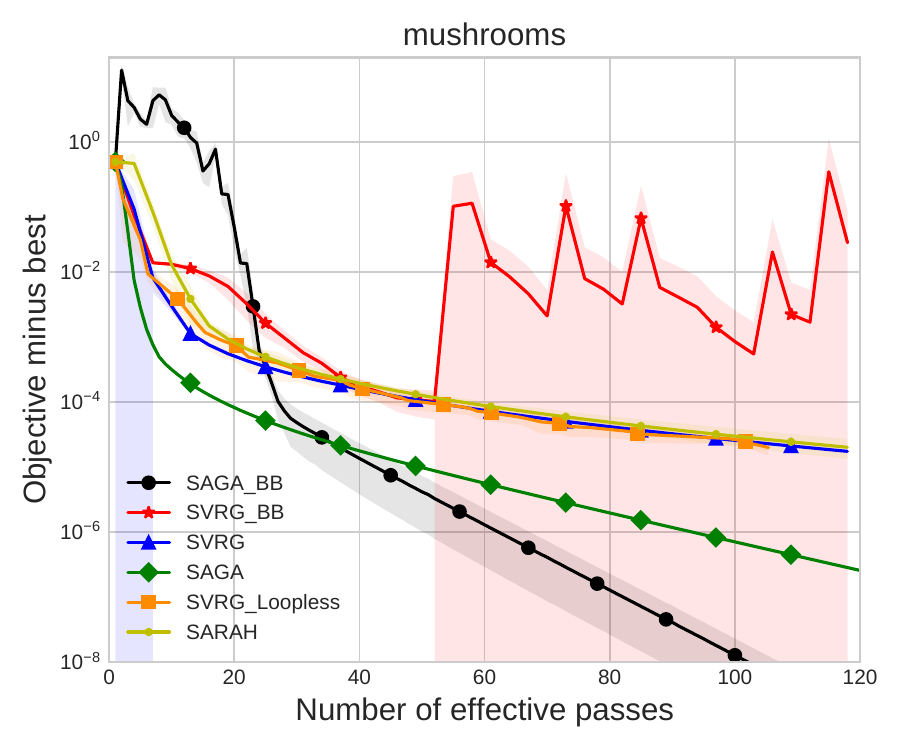}
        \caption{}
    \end{subfigure}\hfill
    \begin{subfigure}[b]{0.235\textwidth}
        \centering
        \includegraphics[width=\linewidth]{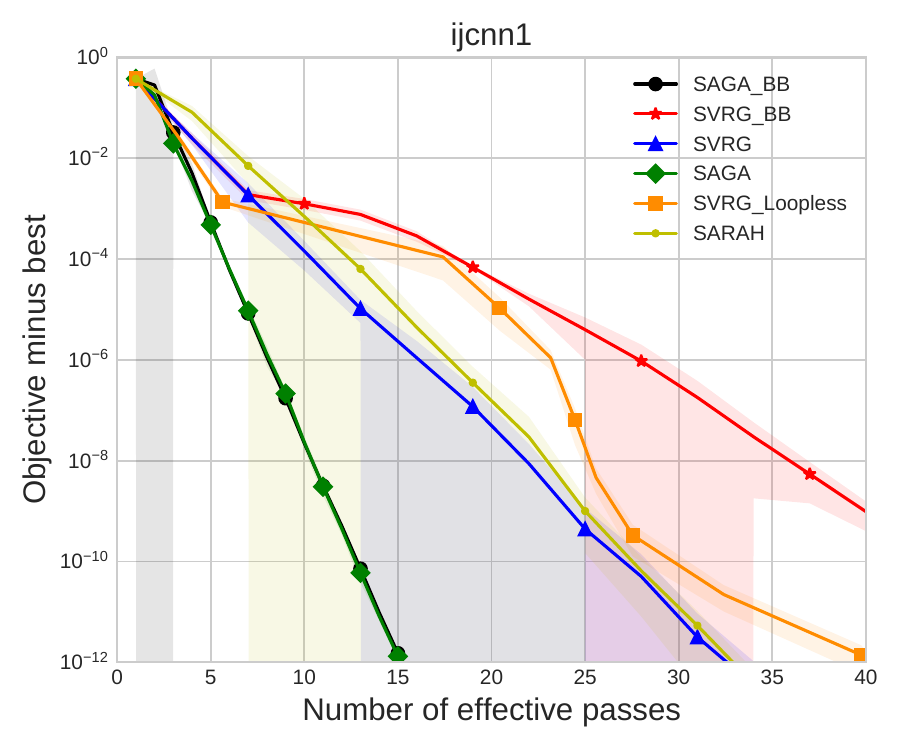}
        \caption{}
    \end{subfigure}\hfill
    \begin{subfigure}[b]{0.235\textwidth}
        \centering
        \includegraphics[width=\linewidth]{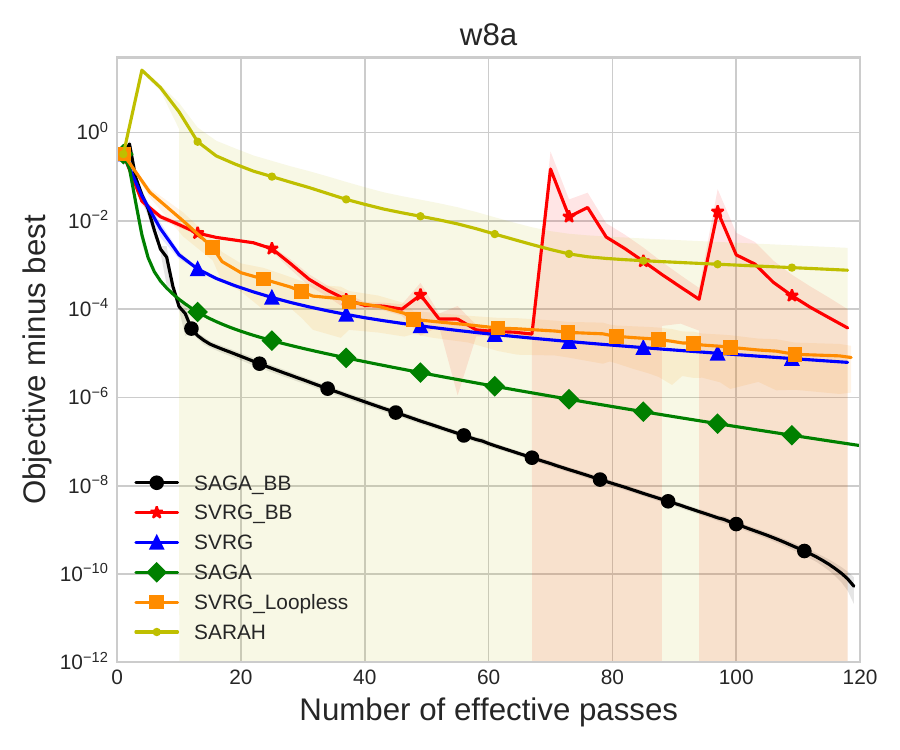}
        \caption{}
    \end{subfigure}\hfill
    \begin{subfigure}[b]{0.235\textwidth}
        \centering
        \includegraphics[width=\linewidth]{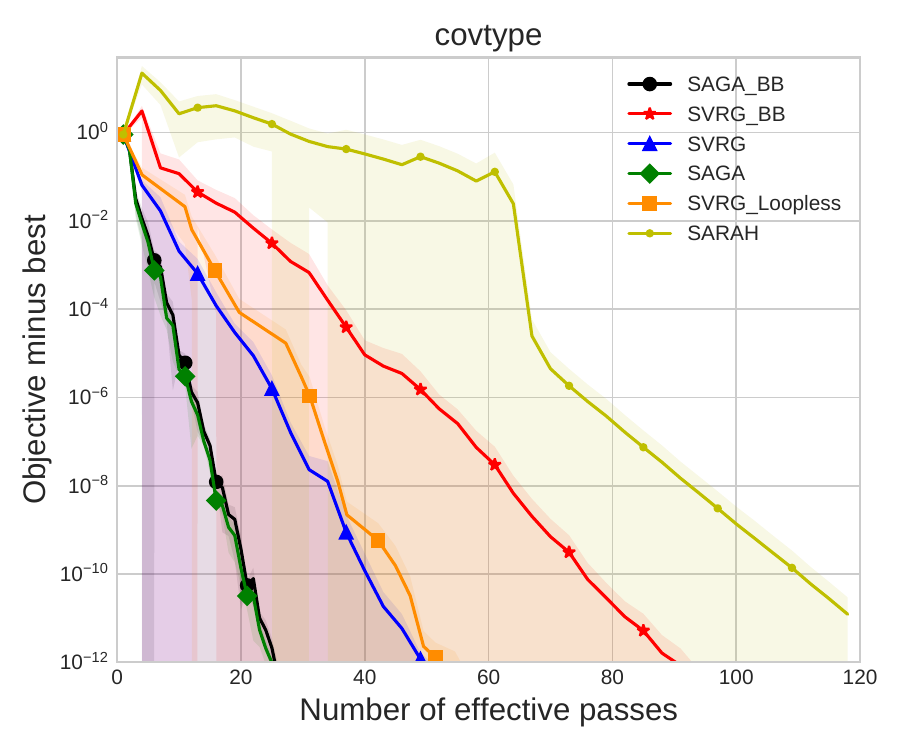}
        \caption{}
    \end{subfigure}
    \caption{Gradient efficiency on four LibSVM datasets with Huber loss. Lines and shaded regions show the mean and standard deviation over five runs.}
    \label{fig: vis huber loss}
\end{figure}

\begin{figure}[tp!]
    \centering
    \begin{subfigure}[b]{0.235\textwidth}
        \centering
        \includegraphics[width=\linewidth]{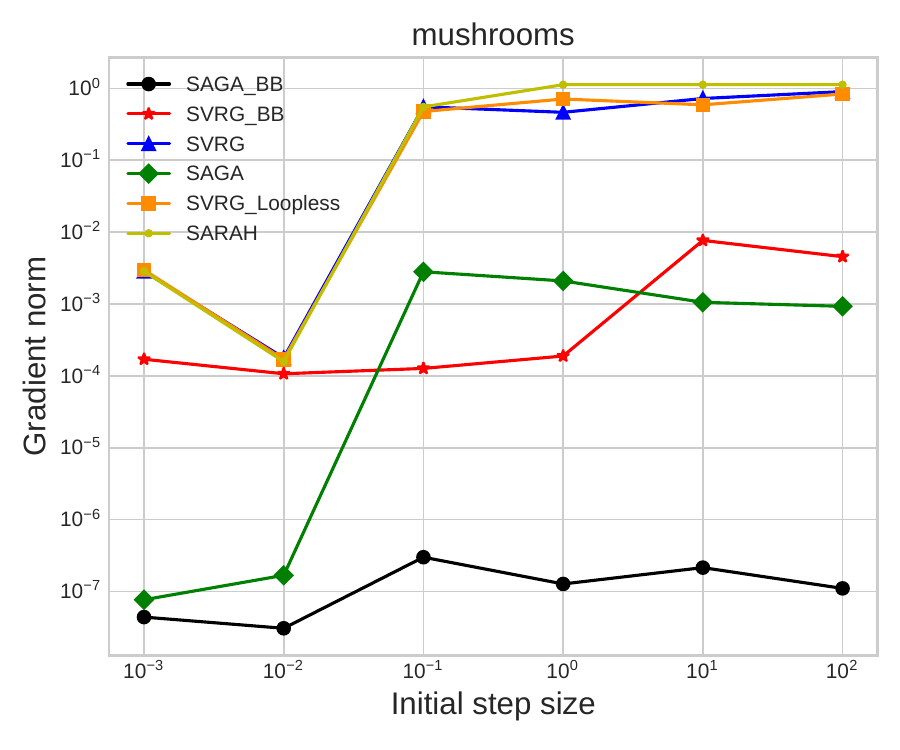}
        \caption{}
    \end{subfigure}\hfill
    \begin{subfigure}[b]{0.235\textwidth}
        \centering
        \includegraphics[width=\linewidth]{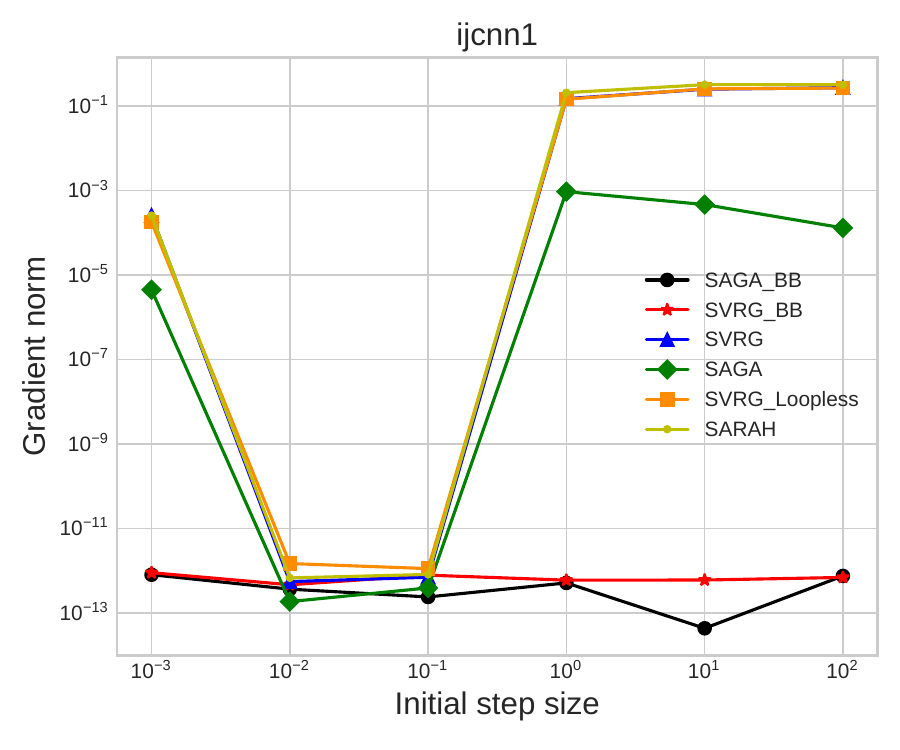}
        \caption{}
    \end{subfigure}\hfill
    \begin{subfigure}[b]{0.235\textwidth}
        \centering
        \includegraphics[width=\linewidth]{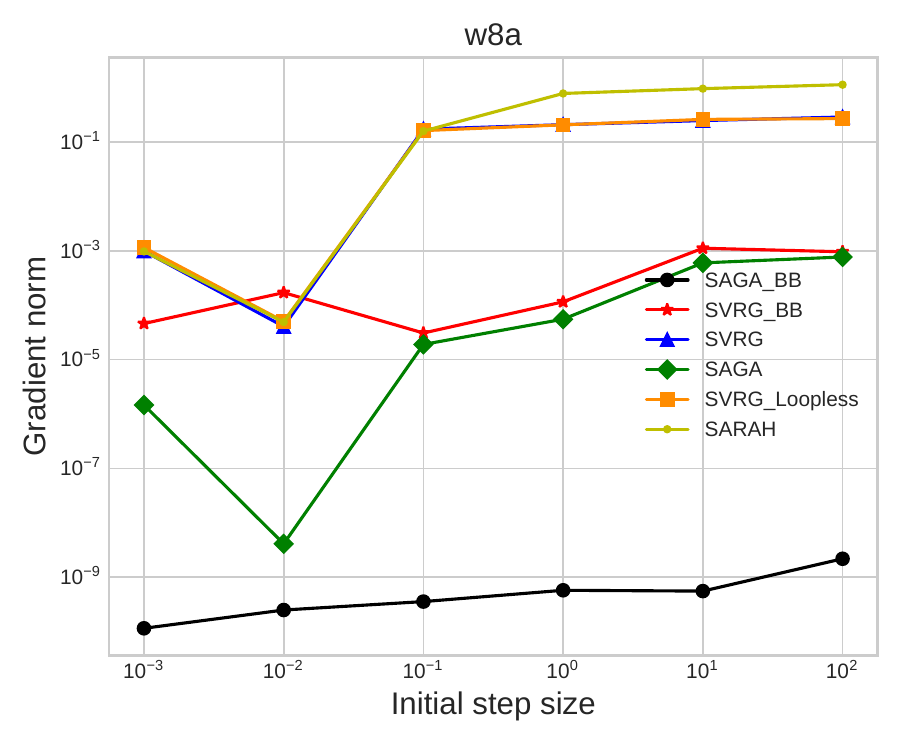}
        \caption{}
    \end{subfigure}\hfill
    \begin{subfigure}[b]{0.235\textwidth}
        \centering
        \includegraphics[width=\linewidth]{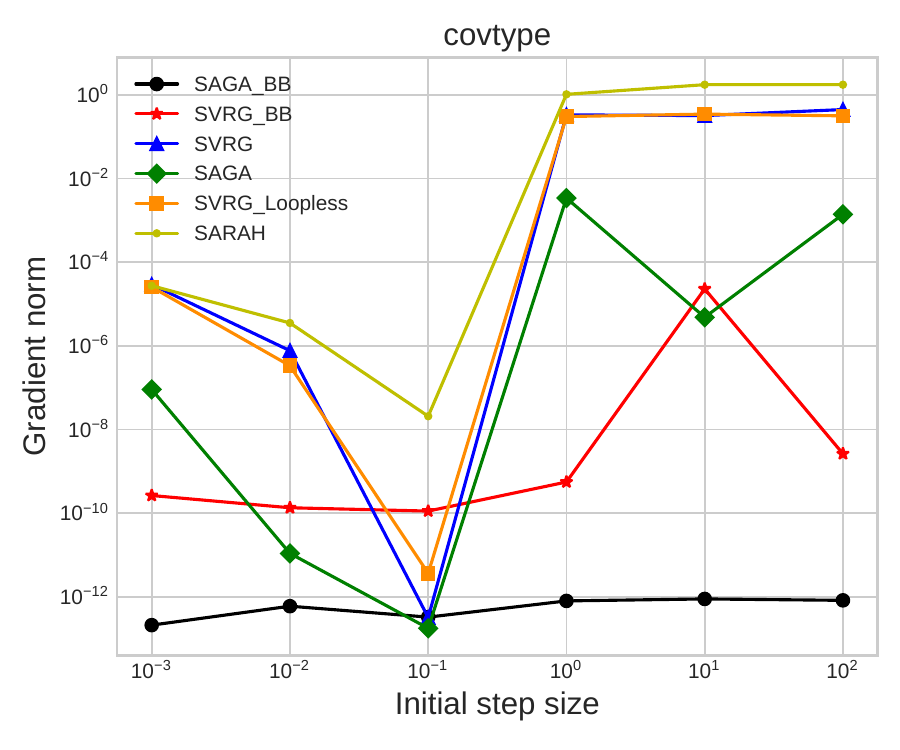}
        \caption{}
    \end{subfigure}
    \caption{Initial-step sensitivity for Huber loss with batch size $b=1$. Gradient norms are capped at $1$ for readability.}
    \label{fig: huber loss bz1}
\end{figure}

\begin{figure}[tp!]
    \centering
    \begin{subfigure}[b]{0.235\textwidth}
        \centering
        \includegraphics[width=\linewidth]{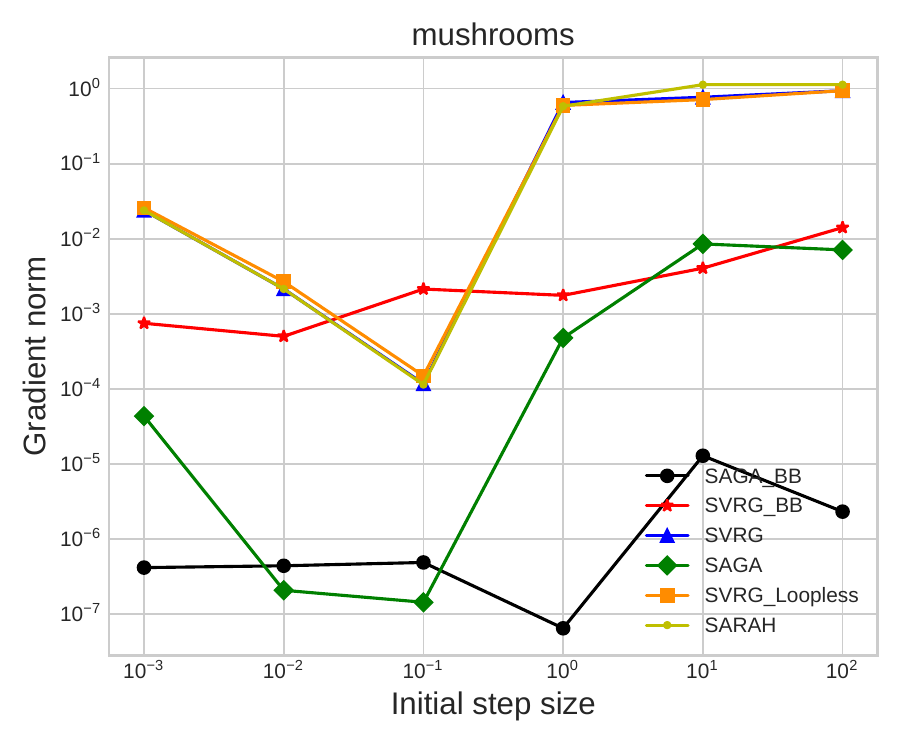}
        \caption{}
    \end{subfigure}\hfill
    \begin{subfigure}[b]{0.235\textwidth}
        \centering
        \includegraphics[width=\linewidth]{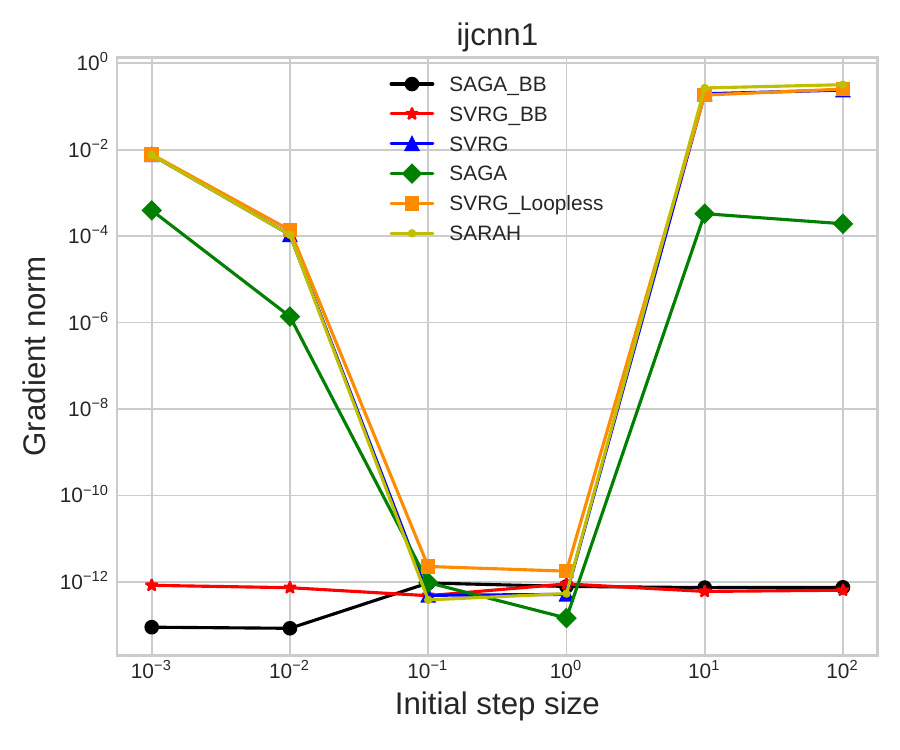}
        \caption{}
    \end{subfigure}\hfill
    \begin{subfigure}[b]{0.235\textwidth}
        \centering
        \includegraphics[width=\linewidth]{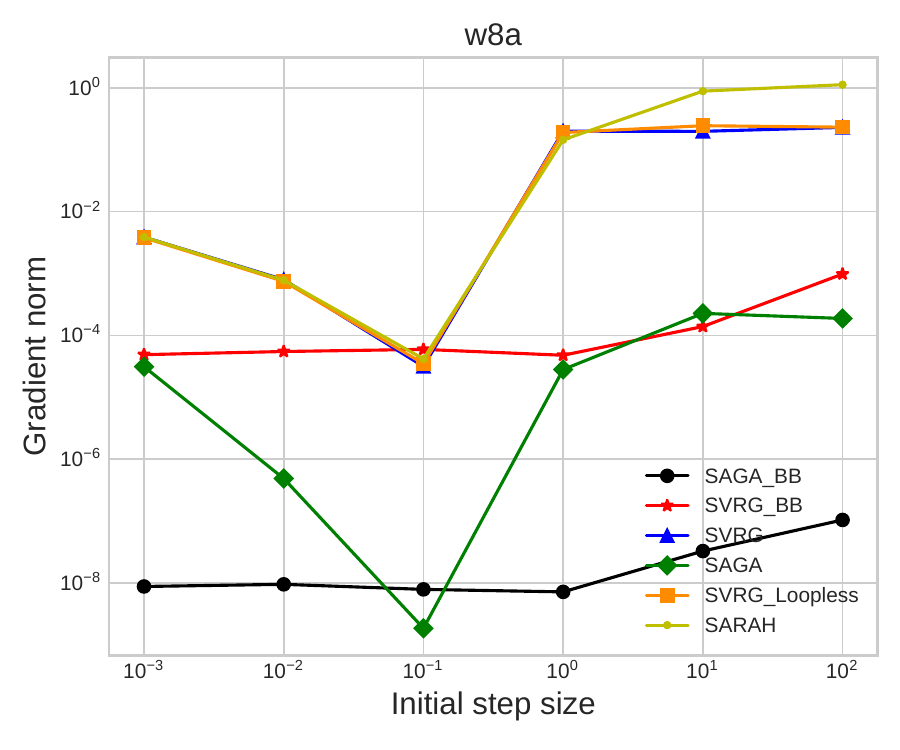}
        \caption{}
    \end{subfigure}\hfill
    \begin{subfigure}[b]{0.235\textwidth}
        \centering
        \includegraphics[width=\linewidth]{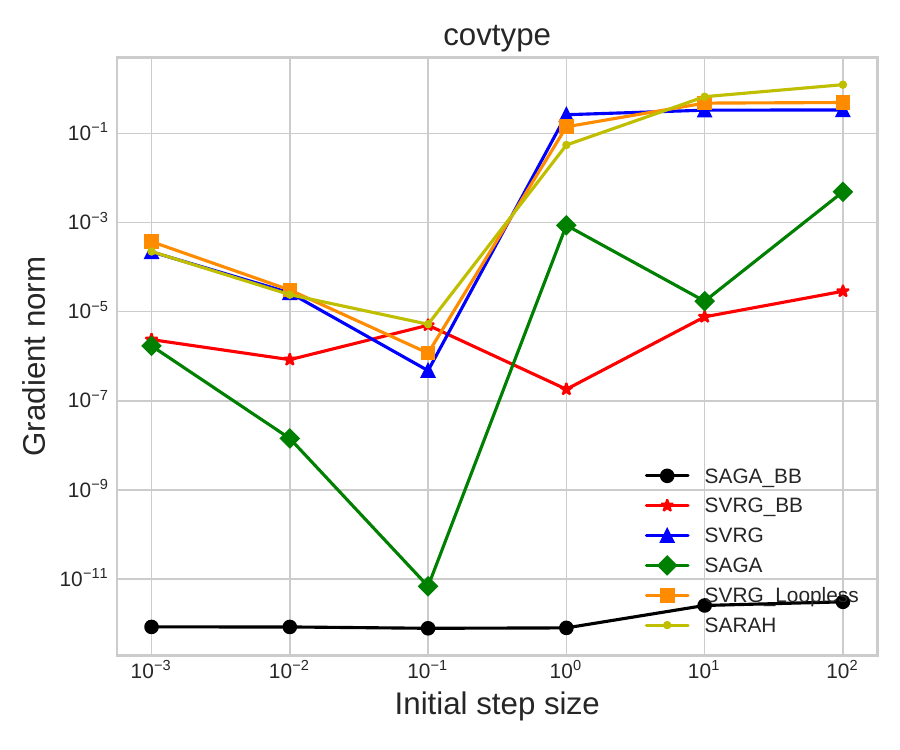}
        \caption{}
    \end{subfigure}
    \caption{Initial-step sensitivity for Huber loss with batch size $b=8$. Gradient norms are capped at $1$ for readability.}
    \label{fig: huber loss bz8}
\end{figure}

\begin{figure}[tp!]
    \centering
    \begin{subfigure}[b]{0.235\textwidth}
        \centering
        \includegraphics[width=\linewidth]{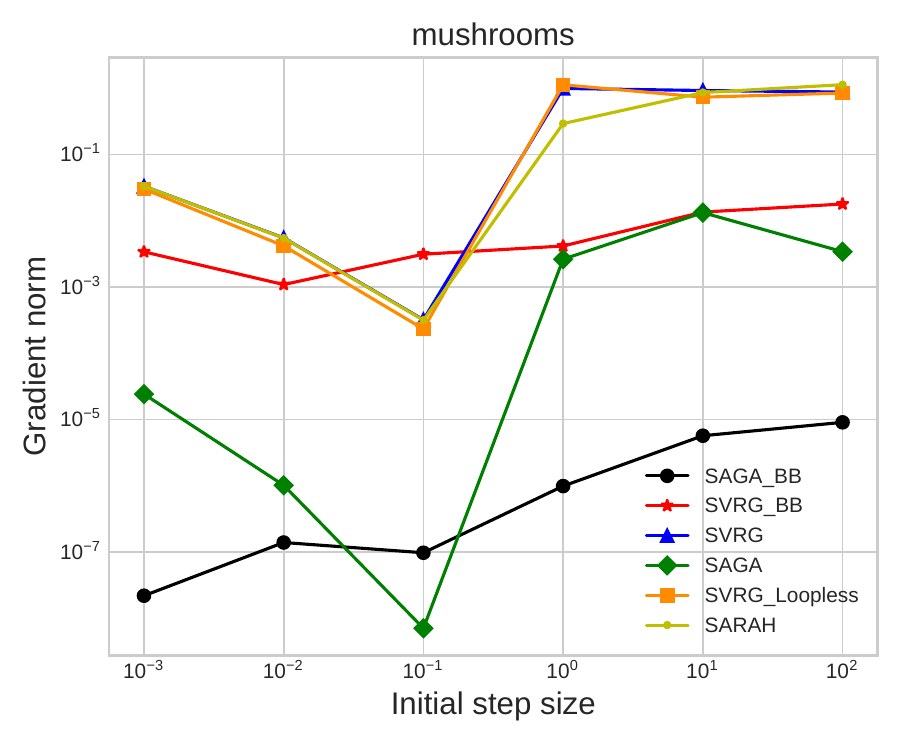}
        \caption{}
    \end{subfigure}\hfill
    \begin{subfigure}[b]{0.235\textwidth}
        \centering
        \includegraphics[width=\linewidth]{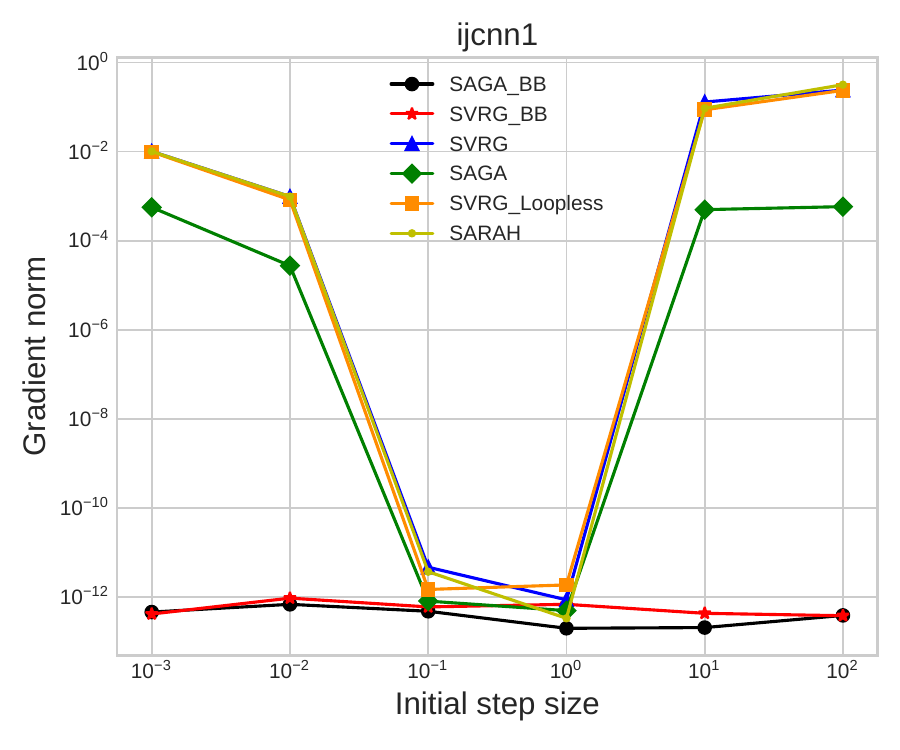}
        \caption{}
    \end{subfigure}\hfill
    \begin{subfigure}[b]{0.235\textwidth}
        \centering
        \includegraphics[width=\linewidth]{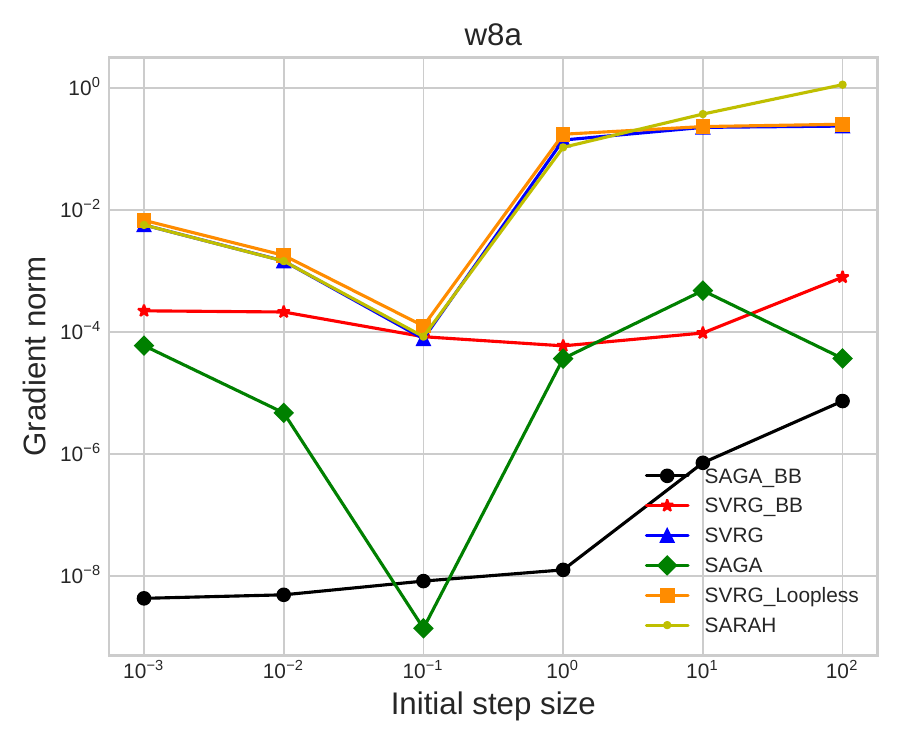}
        \caption{}
    \end{subfigure}\hfill
    \begin{subfigure}[b]{0.235\textwidth}
        \centering
        \includegraphics[width=\linewidth]{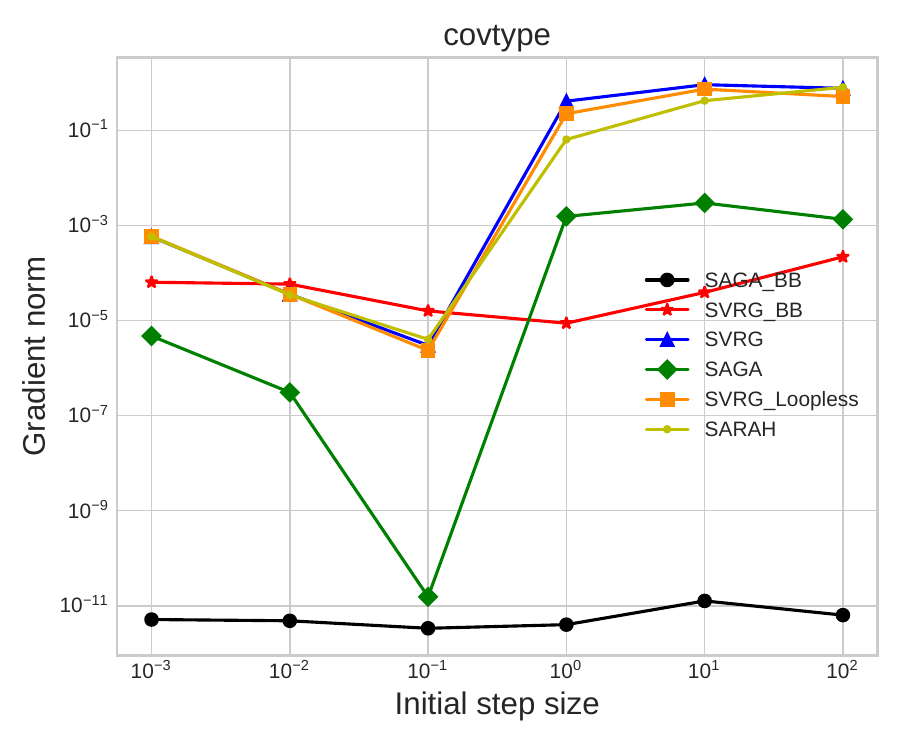}
        \caption{}
    \end{subfigure}
    \caption{Initial-step sensitivity for Huber loss with batch size $b=16$. Gradient norms are capped at $1$ for readability.}
    \label{fig: huber loss bz16}
\end{figure}

\begin{figure}[tp!]
    \centering
    \begin{subfigure}[b]{0.235\textwidth}
        \centering
        \includegraphics[width=\linewidth]{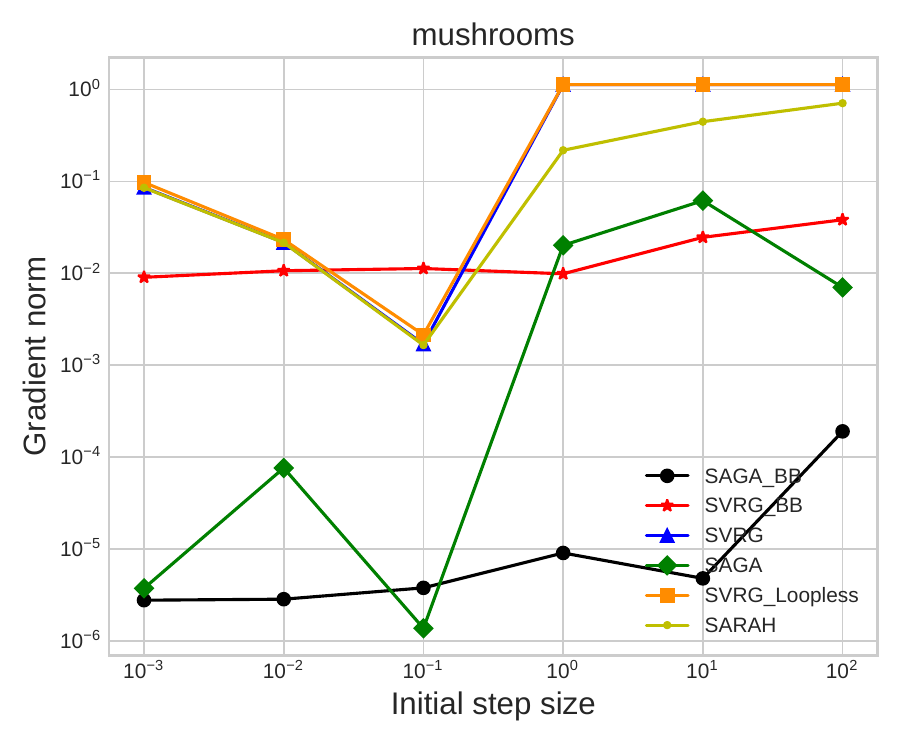}
        \caption{}
    \end{subfigure}\hfill
    \begin{subfigure}[b]{0.235\textwidth}
        \centering
        \includegraphics[width=\linewidth]{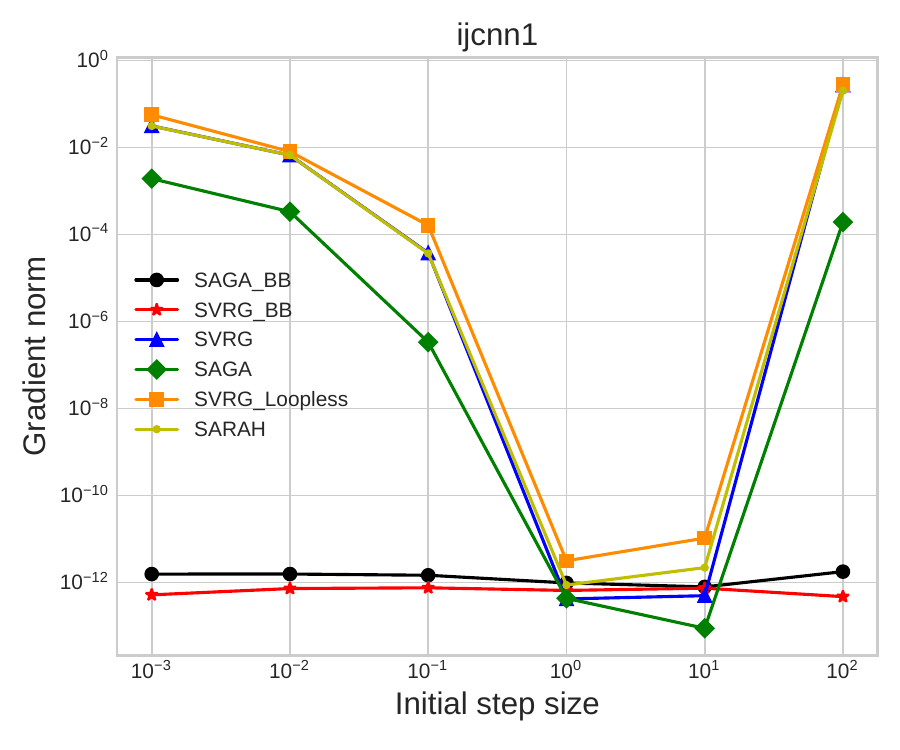}
        \caption{}
    \end{subfigure}\hfill
    \begin{subfigure}[b]{0.235\textwidth}
        \centering
        \includegraphics[width=\linewidth]{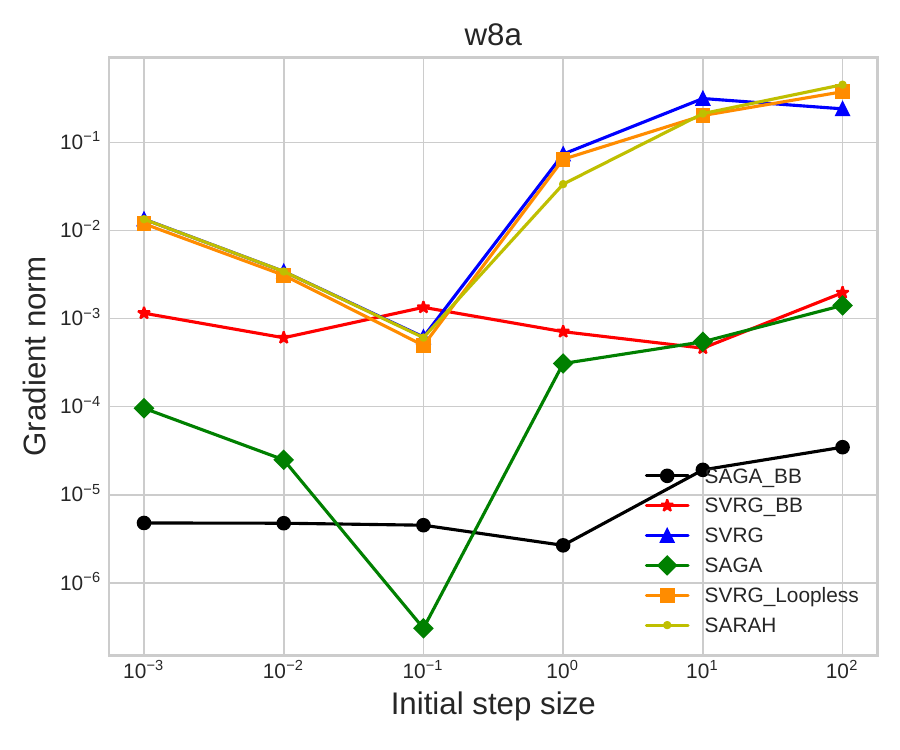}
        \caption{}
    \end{subfigure}\hfill
    \begin{subfigure}[b]{0.235\textwidth}
        \centering
        \includegraphics[width=\linewidth]{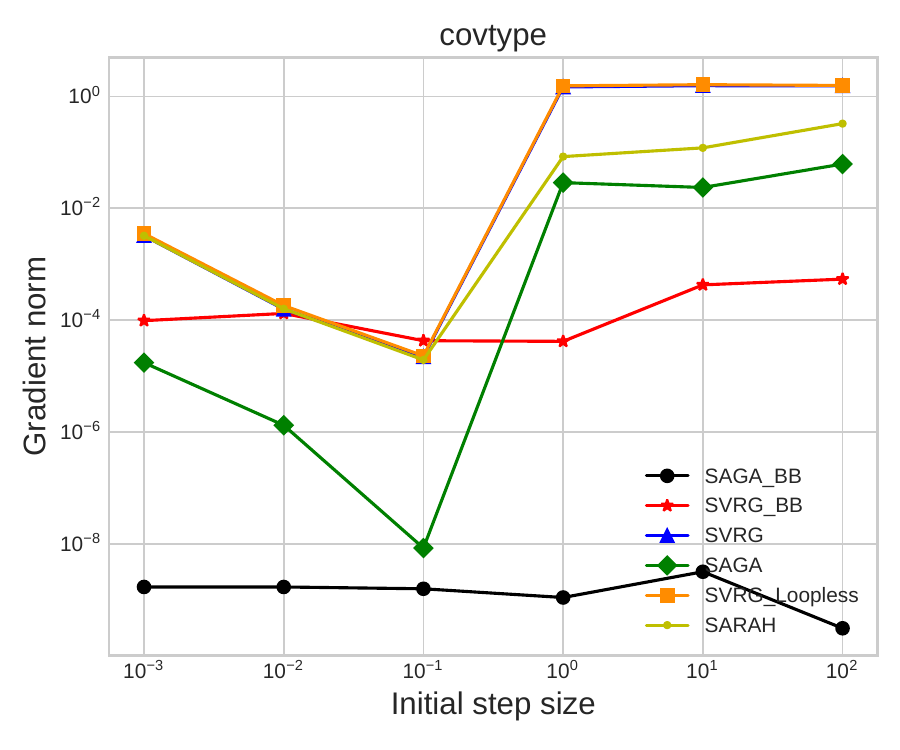}
        \caption{}
    \end{subfigure}
    \caption{Initial-step sensitivity for Huber loss with batch size $b=64$. Gradient norms are capped at $1$ for readability.}
    \label{fig: huber loss bz64}
\end{figure}

Figure \ref{fig: vis huber loss} reproduces the gradient-efficiency pattern observed with logistic loss: SAGA-BB reaches the lowest objective gap in fewer effective passes on each dataset. Figures \ref{fig: huber loss bz1}--\ref{fig: huber loss bz64} test the complementary stability question across four batch sizes; the consistently broad low-gradient region ties the empirical robustness to the stabilized adaptive mechanism rather than to a single logistic-regression configuration.

\end{document}